\newif\ifarxiv
\arxivtrue
\ifarxiv
	\documentclass[11pt]{article}
\usepackage{xcolor,hyperref}
\selectcolormodel{cmyk}%
   \usepackage[margin=3cm]{geometry}
\usepackage{delimdelim,xspace,microtype,lmodern}
\usepackage{mymathenv}

\usepackage{caption}

\usepackage[font=sf]{caption}

\crefformat{equation}{(#2#1#3)}
\usepackage[backgroundcolor=yellow]{todonotes}
\graphicspath{{./figs/}{./Pictures/}}
\usepackage[numbers,sort]{natbib}
\numberwithin{equation}{section} 

\newcommand\coloneq{\coloneqq}
\newcommand\eqcolon{\eqqcolon}
\newcommand\real{\mathbb{R}}

\newcommand\naturals{\mathbb{N}}
\newcommand\trans{\mathsf{T}}    
\newcommand\iid{{\emph{iid}}\xspace}
\renewcommand{\vec}[1]{{\boldsymbol{#1}}} 
\newcommand{\email}[1]{\url{#1}}
\newcommand\cov{\operatorname{Cov}}

\newcommand\Nrm{\operatorname{N}} 
\crefformat{section}{\S#2#1#3}
\crefformat{subsection}{\S#2#1#3}
\newcommand\tstep{{\Delta t}}
\title{Langevin equations for landmark image registration  with uncertainty	\footnote{This work was partially supported by the LMS Scheme 7 grant SC7-1415-09.}}
\author{Stephen Marsland\footnote{Massey University \email{s.r.marsland@massey.ac.nz}} \and Tony
	Shardlow\footnote{Department of Mathematical Sciences, University of Bath, Bath BA2 7AY, UK \email{t.shardlow@bath.ac.uk}}}
\usepackage{lipsum}
\usepackage{amsfonts}
\usepackage{graphicx}
\usepackage{epstopdf}
\usepackage{algorithmic}
\ifpdf
  \DeclareGraphicsExtensions{.eps,.pdf,.png,.jpg}
\else
  \DeclareGraphicsExtensions{.eps}
\fi
\makeatletter
\newenvironment{keywords}{%
	\small	\quotation{\bfseries Keywords: }}%
{\endquotation}
\newenvironment{AMS}{%
	\small	\quotation{\bfseries AMS subject classifications: }}%
{\endquotation}
\makeatother
	\usepackage{natbib,doi}
\else
	\documentclass[review]{siamonline0316}
	\input{sde_imag_shared}
\fi
\begin{document}
\maketitle
\begin{abstract}
Registration of images parameterised by landmarks provides a useful method of describing shape variations by computing the minimum-energy time-dependent deformation field that flows one landmark set to the other. This is sometimes known as the geodesic interpolating spline and can be solved via a Hamiltonian boundary-value problem to give a diffeomorphic registration between images. However, small changes in the positions of the landmarks can produce large changes in the resulting diffeomorphism. We formulate a Langevin equation for looking at small random perturbations of this registration. The Langevin equation and three computationally convenient approximations are introduced and used as prior distributions. A Bayesian framework is then used to compute a posterior distribution for the registration, and also to formulate an average of multiple sets of landmarks.
\end{abstract}


\section{Introduction}

The mathematical description of shape and shape change has become an area of significant research interest in recent years, not least because of its applications in Computational Anatomy, where variations in the appearance of objects in medical images are described mathematically in the hope that their change can be linked to disease progression.
When two images are topologically equivalent, they can be brought into alignment (registered) by deforming one of the images without tearing or folding, so that their appearance matches as closely as possible. This can be formulated mathematically by taking two images $I, J \colon B \to \real$ (for some physical domain $B \subset \real^d$) that act as reference and target respectively. (In medical imaging, these are typically greyscale images.) Image $I$ is then deformed by some diffeomorphism  $\vec\Phi\colon B\to B$ such that $I \circ \vec\Phi^{-1}$ and $J$ are as close as possible according to some model of similarity. In addition to defining similarity, the metric on the diffeomorphism group also has to be selected; the typical setting is to use the right-invariant $H_{\alpha}^1$ metric, which leads to the so-called EPDiff equation~\citep{HolmMarsden05}. We can also define a `bending energy' of $\vec\Phi$ in analogy to the thin-plate spline~\citep{Duchon,Bookstein}. For a general treatment and an overview of the subject, see the monograph \citep{Younes} and references therein.

Similarity can be understood as a norm on the images $\Vpair{I\circ \vec \Phi^{-1}-J}$, in which case a common choice is the sum-of-squares of pixel values, although there are plenty of other options (see e.g.,~\cite{Modersitzki}). Alternatively, similarity can be expressed by a set of landmarks that identify corresponding points on each image. Our focus is on the second of these two methods. Specifically, we consider a set of
landmarks on the reference and target images, $\vec q_i^r$ and $\vec q_i^t$, for $i=1,\dots,N$ in
$B$ and we aim to find $\vec\Phi$ such that $\vec\Phi(\vec q_i^r)=\vec q_i^t$.
Obviously, landmarks need to correspond between the images, and this is a difficulty with landmark-based methods whether the landmarks are selected manually or automatically (for example, by an algorithm that looks for points in the images that should be well-defined such as points of maximum curvature or minimum intensity). In either case, it is easy for errors to be made so that points that should be in correspondence are not, or where there is some random error in the positioning of the landmark with respect to the point it is intended to mark. For humans, marking up points on objects consistently is particularly difficult, and there is experimental evidence that lack of correspondence between pairs of landmarks can substantially effect the diffeomorphisms that are identified in order to match the images, see for example~\citep{Marsland04a}. We provide a solution to this problem based on a Bayesian formulation of the landmark matching problem.


In this paper, we parameterise the diffeomorphisms by
time-dependent deformation fields
$\vec v\colon [0,1]\times B\to \real^d$ and define
$\vec \Phi(\vec Q)=\vec q(1)$ for $\vec Q\in B$, where
$\vec q(t)$ for $t\in[0,1]$ satisfies the initial-value problem
\begin{equation}
\frac{d\vec q}{dt}%
=\vec v(t,\vec q(t)),\qquad %
\vec q(0)=\vec Q.\label{eq:a}
\end{equation}
The bending energy of $\vec \Phi$ is defined via a norm on the deformation field:
\begin{equation}
\operatorname{Energy}(\vec \Phi)%
\coloneq\frac12\norm{\vec v}^2,\qquad
\norm{\vec v}\coloneq\pp{ \int_0^1 \norm{{\cal L}
    \vec v(t,\cdot)}_{L^2(B, \real^d)}^2\,dt}^{1/2},\label{eq:normphi:normphi}
\end{equation}
for a differential operator ${\cal L}$ (for example,
$\mathcal{L}$ equals the Laplacian $\Delta$ with clamped-plate boundary
conditions~\citep{marsland2002clamped}).

The case where landmarks are fully observed is well-studied
and the solution is given by the following boundary-value
problem: let $G$ be the Green's function associated to the
operator ${\cal L}^2$, and let $\vec p_i(t), \vec q_i(t)$
satisfy the Hamiltonian boundary-value problem
\begin{equation}\label{eq:ham}
\frac{d\vec p_i}{dt}%
=-\nabla_{\vec q_i}H,\qquad%
\frac{d\vec q_i}{dt}%
=\nabla_{\vec p_i} H,
\end{equation}
subject to $\vec q_i(0)=\vec q_i^r$ and $\vec q_i(1)=\vec
q_i^t$ for the Hamiltonian $H\coloneq\frac 12\sum_{i,j=1}^N
\vec p_i^\trans \vec p_j G(\vec q_i,\vec q_j)$. Here $\vec
p_i$ are known as generalised momenta. The diffeomorphism
$\vec\Phi$ is now defined by \cref{eq:a} with
\begin{equation}%
\label{eq:a1}%
\vec v(t,\vec q)%
=\sum_{i=1}^N \vec p_i(t)%
G(\vec q,\vec q_i(t)).
\end{equation}
In general, $G$ is defined directly rather than by specifying
the Green's functions of a known $\mathcal{L}$. In our
experiments, we take the Gaussian function
$G(\vec q_1,\vec q_2)=\exp(-(\Vpair{\vec q_1-\vec
  q_2}/\ell)^2)$ for a length scale $\ell$.  For smooth
choices of $G$ such as this, $\vec\Phi$ is a continuously
differentiable function. It is invertible by
reversing the direction of the flow and hence
$\vec\Phi\colon B\to B$ is a diffeomorphism. See for example
\citep{Marsland2006} and, in more general situations,
\citep{McLachlan07a,Holm04}.

Our focus in this paper is to treat uncertainty around landmark positions and
sensitivity of the diffeomorphism to noise. To study this
problem, we introduce a Bayesian formulation and define
prior distributions on the set of diffeomorphisms. We then
condition the prior on noisy observations of the
landmarks to define a posterior distribution. 

The choice of prior distribution is an important consideration, and we make a practical choice
that ensures that diffeomorphisms that have less bending energy are preferred.  This is
the Gibbs canonical distribution, which also has the benefits that both ends of the path are treated equally
and it has a time reversal symmetry (i.e.,  the Gibbs distribution is invariant under change of variable $t\mapsto 1-t$).

We consider Langevin-type perturbations of \cref{eq:ham}, which
have the Gibbs distribution $\exp(-\beta H)$ (with inverse
temperature $\beta>0$) as an invariant measure. The advantage
now is that, with suitable initial data, the solutions of
the Langevin equation $[\vec p_i(t), \vec q_i(t)]$ all
follow the same distribution $\exp(-\beta H)$ for
$t\in[0,1]$.  

It can be seen that diffeomorphisms with lower bending energy are preferred by considering the Hamiltonian using \cref{eq:a1}:
\begin{align*}
  H(\vec p_i(t),\vec q_i(t))%
  =&\frac 12  \sum_{j=1}^N \vec p_j(t)^\trans \vec v(t, \vec q_j(t))
     =\frac 12 \sum_{j=1}^N \int_B \vec p_j(t)^\trans \delta_{\vec
     q_j(t)}(\vec x)\vec v(t, \vec x)\,d\vec x\\ %
  \intertext{(if $\mathcal{L}^2 G=\delta$ and $\mathcal{L}$ is self
  adjoint)}%
  =&\frac 12   \ip{ \mathcal{L}^2 \vec v(t,\cdot),%
     \vec v(t, \cdot)}_{L^2(B,\real^d)} %
     = \frac 12  \norm{\mathcal{L}
     \vec v(t,\cdot)}^2_{L^2(B,\real^d)}.
\end{align*}
Hence, $\int_0^1 H(\vec p_i(t),\vec q_i(t))\,dt= \operatorname{Energy}(\vec \Phi)$ and  we see that
diffeomorphisms $\vec\Phi$ with less bending energy  are associated to paths $[\vec p_i(t), \vec q_i(t)]$ that have a larger
density  under the Gibbs measure $\exp(-\beta H)$.

\subsection{Previous work}\label{sec:pw}

We are aware of three papers that have looked at image
registration in the presence of noise. The most similar to
ours is \citep{trouve10:_shape}, who imagine that the
trajectories $\vec q_i(t)$, for $t\in[0,1]$ and $i=1,\dots,N$,
are noisy observations of some true trajectories $\vec
Q_i(t)$. Specifically, they wish to minimise
\[
\int_0^1 \norm{\mathcal{L}\vec v(t,\cdot)}_{L^2(B,\real^d)}^2\,dt%
+{\sigma}\sum_{i=1}^N \int_0^1 \norm{\vec q_i(t)-\vec Q_i(t)}^2\,dt,
\]
for a parameter $\sigma>0$.
The first term corresponds to a bending energy and second
penalises deviations from $\vec Q_i(t)$. This leads to a
controlled Hamiltonian system
\[
\frac{d\vec p_i}{dt}%
= -\nabla_{\vec q_i} H+ \sigma(\vec q_i-\vec Q_i(t)),\qquad
\frac{d\vec q_i}{dt}%
= \nabla_{\vec p_i} H.
\]
If a white-noise model is assumed for the observation error
$\vec q_i(t)-\vec Q_i(t)$, this gives the SDE
\begin{equation}\label{tush}
{d\vec p_i}%
= -\nabla_{\vec q_i} H\,dt+ \sigma \,d\vec W_i(t),\qquad
\frac{d\vec q_i}{dt}%
= \nabla_{\vec p_i} H.
\end{equation}
This system is identical to \cref{eq:11}, except that no
dissipation is included and therefore it will not have a Gibbs'
distribution as invariant measure.

In \citep{MR3042087} registrations where curves are
  matched (in two dimensions) are studied. A set of discrete points is
  defined on one curve and noisy observations are made on the
  second. Registrations are defined by an initial
  momenta and, to match curves rather than points,
  reparameterisations of the curve are also included. A
  Gaussian prior distribution is defined on the joint space
  of initial momenta and reparameterisations. Observations
  are made with independent Gaussian noise. The authors
  provide an MCMC method for sampling the posterior
  distribution. Hamiltonian equations are used to define the
  diffeomorphism and no noise is introduced along the
  trajectories. In the case of landmark matching, there is
  no advantage to introducing a prior distribution on the
  initial momentum as the data specifies the initial
  momentum completely. For noisy landmark matching, the
  approach has value, being simpler than the Langevin
  equations, but the results will depend on which end the
  prior distribution is specified.

A method to include stochasticity into the Large Deformation Diffeomorphic Metric Mapping (LDDMM) 
framework of image registration (see \citep{Younes} for details) is presented in \citep{Arnaudon16}. In this
approach, noise is introduced into the time-dependent deformation field 
from the start point to the end point, leading to a stochastic version of the EPDiff equations. The authors
also introduce an EM algorithm for estimating the noise parameters based on data. 
The approach is based on two other papers of relevance, which add cylindrical noise to
the variational principles of systems of evolutionary PDEs. By taking
the system in variational form, this introduces noise perturbations
into the advection equation (which corresponds to \cref{eq:a1}). To preserve the conservation laws encoded in the PDEs, the
$\vec p$ update equations are left unchanged.  The resulting trajectories
in $\vec q_i(t)$ have the same regularity as Brownian motion and satisfy Stratonovich SDEs, which are invariant to the relabelling Lie group.
The approach was originally
developed for the Euler equations for an ideal fluid in~\citep{Holm15}, and
was extended to the Euler--Poincar\'{e} (EPDiff) equations in~\citep{Holm16}.
While their examples are for soliton dynamics in one spatial dimension, under
particular choices of metric on the diffeomorphism group, the equations of
image deformation are also EPDiff equations, hence the work in  \citep{Arnaudon16}. 



%

\subsection{Organisation}
This paper is organised as follows.  Our Langevin equations
are described in \S\ref{sec:lang} and some basic theory
established. Unfortunately, these Langevin equations are hypoelliptic and the Hamiltonian is not separable, making
the equations
difficult to work with numerically. Therefore, in
\S\ref{sec:appr_lang}, we introduce three numerically
convenient prior distributions based on the Langevin
equation.  \S\ref{sec:bayes} formulates inverse
problems based on the prior distributions. Two are image
registrations given noisy observations of the landmarks; the
other asks for the average position of a family of
landmark sets. This section includes numerical
experiments demonstrating our method on a variety of
simple curve registrations. Further simulations and examples
are given in the Supplementary Material.

\subsection{Notation} \label{notation}
We denote the Euclidean norm on $\real^d$ by
$\Vpair{\vec{x}}=\sqrt{\vec x^\trans \vec x}$ and the $d\times d$ identity matrix by $I_d$. For a subset
$B$ of $\real^d$, $L^2(B,\real^d)$ is the usual Hilbert space of
square-integrable functions from $B\to \real^d$ with inner product
$\ip{\vec f, \vec g}_{L^2(B,\real^d)}=\int_B \vec
f(\vec x)^\trans \vec g(\vec x)\,d\vec{ x} $. We often work
with position vectors $\vec q_i\in B \subset \real^d$ and
conjugate momenta $\vec p_i\in \real^d$ for
$i=1,\dots,N$. We denote the joint vector $[\vec
p_1,\dots,\vec p_N]$ by $\vec p\in \real^{dN}$ and similarly
for $\vec q\in\real^{dN}$. The combined vector $[\vec p,\vec
q]$ is denoted $\vec z\in \real^{2dN}$. For a symmetric and positive-definite function $G\colon \real^d\times \real^d\to \real$, let $\mathcal{G}(\vec q)$ denote the $N\times N$ matrix with entries $G(\vec q_i,\vec q_j)$.

\section{Generalised Langevin equations}\label{sec:lang}

The classical landmark-matching problem can be solved as
a Hamiltonian boundary-value problem. The dynamics in a
Hamiltonian model have constant energy as measured by
$H$. Instead, we connect the system to a heat bath and %
look at constant-temperature dynamics. We consider a heat bath
with inverse temperature $\beta$. One method of
constant-temperature particle dynamics is the Langevin
equation. That is, we consider the system of stochastic ODEs on $\real^{2dN}$
given by
  \begin{equation}
    \label{eq:11}
    d\vec p_i%
    = \Big[-\lambda \nabla_{\vec p_i} H- \nabla_{\vec
      q_i}H\Big]\,dt%
    +
    \sigma\, d \vec W_i(t),\qquad
    \frac{d\vec q_i}{dt}%
    =  \nabla_ {\vec p_i}H
  \end{equation}
  for a dissipation $\lambda>0$ and diffusion
  $\sigma>0$. Here $\vec W_i(t)$ are \iid $\real^d$ Brownian
  motions.  For $\beta=2\lambda/\sigma^2$,  a potential $V\colon \real^{d N}\to \real$,  and
  $H=\frac 12 \vec p^\trans \vec p +V(\vec q)$,
  \cref{eq:11} is the classical Langevin equation where the
  marginal invariant distribution for $\vec p$ is
  $\Nrm(\vec 0,\beta^{-1} I_{dN})$ and hence the average temperature
  $\frac 1d\mean{\vec p_i^\trans \vec p_i}$ per degree of
  freedom is the constant $\beta^{-1}$.  Let
  $[\vec p_i(t), \vec q_i(t)]$ for $t\in[0,1]$ satisfy
  \cref{eq:11} and define $\vec\Phi(\vec Q)$ as in
  \cref{eq:a,eq:a1}. Notice that
  $\vec \Phi(\vec q_i(0))=\vec q_i(1)$.  In perturbing
  \cref{eq:11} from \cref{eq:ham}, only the momentum
  equation is changed, so the equations for $\vec q$ are
  untouched and are consistent with the definition of
  $\vec v(t,\vec q)$ and hence $\vec \Phi$.

  The solution of \cref{eq:11} is related to \cref{tush} by
  a Girsanov transformation. Let $\pi$ and $\nu$ be the
  distribution on the path space $C([0,1],\real^{2dN})$ of
  \cref{eq:11} and \cref{tush} respectively. Then, for
  $\vec {z}=[\vec{ p},\vec {q}	]$,
  \[
    d\pi(\vec z)=\frac{1}{\phi(\vec z)} d\nu(\vec z),
  \]
  where
  \[
    \log(\phi(\vec z))%
    =\sum_{i=1}^N \bp{\frac{\lambda}{\sigma}\int_0^1  \vec p_i(t)^\trans  d\vec W_i(t)%
      -\frac {\lambda^2}{2\sigma^2} \int_0^1  \norm{\vec p_i(t)}^2\,dt};
  \]
  see \citep[Lemma 5.2]{Hairer2007-sv}.

  To define a distribution on the family of diffeomorphisms,
  it remains to choose initial data. If we specify a
  distribution on $[\vec p, \vec q]$ at $t=0$,
  \cref{eq:11} implies a distribution on the paths and hence
  on $\vec \Phi$ via \cref{eq:a,eq:a1}. The obvious choice
  is the Gibbs distribution
  $\exp(-\beta H)$. If $\sigma^2 \beta=2\lambda$ (the
  fluctuation--dissipation relation), then the Gibbs distribution is an invariant
  measure of \cref{eq:11}. To see this, the generator of
  \cref{eq:11} is
\[
{L}%
=\nabla_{\vec p}H\cdot\nabla_{\vec q}
+(-\lambda \nabla_{\vec p}H-\nabla_{\vec q}H)\cdot\nabla_{\vec p}%
+ \frac 12  \sigma^2 \nabla^2_{\vec p}
\]
and its adjoint
\[
{L}^*\rho%
=-\nabla_{\vec q}\cdot((\nabla_{\vec p}H) \rho)
-\nabla_{\vec p}\cdot((-\lambda \nabla_{\vec p}H-\nabla_{\vec q}H)\rho)
+ \frac 12  \sigma^2 \nabla^2_{\vec p}\rho.
\]
The Fokker--Planck equation for the pdf $\rho(\vec p, \vec q, t)$ is
\begin{align*}
\frac{\partial \rho}{\partial t}%
&= -\nabla_{\vec q} \rho
\cdot \nabla_{\vec p} H
+\pp{\lambda
 \nabla_{\vec p} H\cdot \nabla_{\vec p}
+ \nabla_{\vec p}\cdot \lambda\nabla_{\vec p} H} \rho
+\nabla_{\vec p}\rho \cdot \nabla_{\vec q} H
+\frac12 \sigma^2 \nabla_{\vec p}^2  \rho.
\end{align*}
Put $\rho=e^{-\beta H}$, to see
\begin{align*}
\frac{\partial \rho}{\partial t}%
&=
(\beta\nabla_{\vec q} H)\cdot
\nabla_{\vec p}H\, \rho%
+(-\lambda\nabla_{\vec p} H\cdot
\beta\nabla_{\vec p} H +  \nabla_{\vec p}\cdot\lambda\nabla_{\vec p}H)\rho
-\beta \nabla_{\vec p}H \cdot \nabla_{\vec q} H \rho \\
&\qquad+ \frac 12
 \sigma^2(-\beta \nabla^2_{\vec p}H + \beta^2 \nabla_{\vec p}
H \cdot \nabla_{\vec p} H)\rho.
\end{align*}
Then, $\partial \rho/\partial t=0$ if
$\sigma^2 \beta=2\lambda$ and $\rho$ is an invariant
measure. In some cases, it can be shown additionally that
$\rho$ is a probability distribution.  When $B$ is bounded
(as is usually the case for images), the phase space is
compact in position space and, if  $G$ is a
uniformly positive-definite function, $\exp(-\beta H)$ can
be rescaled to be a probability measure.
This happens for the clamped-plate Green's
function \citep{marsland2002clamped}.  Furthermore, in some
cases, the system is ergodic; precise conditions are given in  \citep{soize94:_fokker}, which
studies generalised Langevin equations such as \cref{eq:11}
and provides conditions on $H$ to achieve a unique invariant
measure. 



While invariant measures are appealing,
we view the trajectories as convenient parameterisations of
the diffeomorphism and are not themselves of interest.
Furthermore, in some cases (see Section 9.2 of \citep{Younes}), the domain $B$ is taken
to be $\real^d$ and $G$ is translation
invariant (e.g.,
$G(\vec q_1,\vec q_2)=\exp(-(\Vpair{\vec q_1-\vec
  q_2}/\ell)^2)$ for a length scale $\ell$) and this means
$\exp(-\beta H)$ cannot be a probability measure on
$\real^{2dN}$.  It is simpler to ask for a distribution on
the diffeomorphism that is invariant under taking the
inverse; that is, $\vec \Phi$ and $\vec \Phi^{-1}$ have the
same distribution. To achieve this,
$[\vec p(t), \vec q(t)]$ should have the same distribution
under the time reversal $t\mapsto 1-t$. This can be achieved
 simply by setting initial data at $t=1/2$ and flowing
forward and backward using the same dynamics. Precisely,
choose an initial probability distribution $\mu^*$ on
$\real^{2dN}$. Given
$[\vec p(1/2),\vec q(1/2)]\sim \mu^*$, compute
$\vec p(t)$ and $\vec q(t)$ for $t>1/2$ by solving
\cref{eq:11}. For $t<1/2$, solve
  \begin{equation}
    \label{eq:11c}
     d\vec p_i%
    = \Big[\lambda \nabla_{\vec p_i} H- \nabla_{\vec
      q_i}H\Big]\,dt%
    +
    \sigma \,d \vec W_i(t),\qquad
    \frac{d\vec q_i}{dt}%
    =  \nabla_ {\vec p_i}H.
  \end{equation}
  Here the sign of the dissipation is changed as we evolve
  the system forward by decreasing $t$. The
  distribution of $[\vec p(t),\vec q(t)]$ is unchanged by
  $t\mapsto 1-t$, as can be verified using the
  Fokker--Planck equation.

  One choice for $\mu^*$ comes by choosing  distinguished landmark positions
  $\vec q_i^*$ and conditioning the Gibbs
  distribution on $\vec q_i^*$.  Define the covariance matrix $C$ by
  $C^{-1}= \beta \mathcal{G}(\vec q^*)\otimes I_d$ (the
  matrix $C$ is positive definite if $G$ is a positive-definite function and the points are distinct; see \cref{notation} for a definition of $\mathcal{G}$). With $\vec q^*\coloneq [\vec q^*_1,\dots,\vec q^*_N]$, we could choose
  $\mu^*
  =\Nrm(\vec 0, C)\times \delta_{\vec q^*}  \asymp\exp(-\beta H(\cdot, \vec q^*))\times \delta_{\vec q^*} $, which is the Gibbs distribution conditioned on positions $\vec q^*$.  We prefer to allow  deviation in the position also, and set $\mu^*= \Nrm(\vec 0,C)\times\Nrm(\vec q^*, \delta^2 I_{dN})$ for
  some variance $\delta^2>0$. Then $\mu^*$ is the product of Gaussian
  distributions, where positions are easily sampled independently from
  $\Nrm(\vec q_i^*,\delta^2 I_d)$ and momenta $\vec p$ are
  sampled from $\Nrm(\vec 0, C)$. The matrix $C$ is a
  $ dN\times dN$-covariance matrix. Despite the size,
  standard techniques such as the Cholesky or spectral factorisation can
  be used to sample $\vec p$.

  To summarise, we have defined two prior distributions, both based on the generalised Langevin system \cref{eq:11c}. Ideally, we take the Gibbs distribution for initial data and flow forward \cref{eq:11c} to define a distribution on $\vec\Phi\colon B\to B$. This approach is not always convenient, as the Gibbs distribution may not be a probability distribution and may also be difficult to sample and calculate with. An alternative is to chose a convenient distribution at $t=1/2$ and flow forward by \eqref{eq:11} and backward by \eqref{eq:11c} to define a distribution on paths and hence on $\vec\Phi$.




\subsection{Push-forward example}
The generalised Langevin equation defines a distribution on
the family of diffeomorphisms $\vec\Phi\colon B\to B$. We choose landmarks $\vec q_1,\dots,\vec q_N$,  the inverse temperature $\beta$,  the dissipation coefficent $\lambda$, and an
initial probability distribution $\mu^*$ on $\real^{2dN}$ at
some time $t^*\in [0,1]$. Then, the Langevin equation can be
solved to find paths $\vec q_i(t), \vec p_i(t)$ for
$t\in[0,1]$ and this defines $\vec\Phi\colon B\to B$ by
\cref{eq:a,eq:a1}.

To numerically simulate \cref{eq:11} with a time step
$\tstep=1/N_\tstep$ for $N_\tstep\in\naturals$, consider
times $t_n=n\tstep$ and the approximation
$\vec P_n\approx [\vec p_1(t_n), \dots,\vec
p_N(t_n)]$ and
$\vec Q_n\approx [\vec q_1(t_n), \dots,\vec
q_N(t_n)]$ given by the Euler--Maruyama method
\begin{equation}
\begin{pmatrix}
\vec P_{n+1}\\\vec Q_{n+1}
\end{pmatrix}
=
\begin{pmatrix}
\vec P_{n}\\\vec Q_{n}
\end{pmatrix}
+\begin{pmatrix}
 -\lambda \nabla _{\vec p} H\tstep-\nabla_{\vec q} H\tstep + \sigma \Delta\vec { W}_n\\
\nabla_{\vec p} H \tstep
\end{pmatrix},\label{eq:euler}
\end{equation}
where $H$ on the right-hand side is evaluated at $(\vec P_n,\vec Q_n)$ and $\Delta \vec W_n\sim \Nrm(\vec 0, I_{dN} \tstep)$ \iid. This method
converges in the root-mean-square sense with first order
(subject to smoothness and growth conditions on $H$) \cite{Kloeden1992-qd}.

We give numerical examples of the
push-forward map $\vec \Phi$  for the Green's function
$G(\vec q_1, \vec q_2)=\exp(-(\Vpair{\vec q_1-\vec
  q_2}/\ell)^2)$ with $\ell=0.5$ in two dimensions
($d=2$). Consider $B=[-1,1]^2$ and twenty regularly spaced
reference points $\vec q_i^r$ on the unit circle. For the
initial distribution, we take $\vec q_i(0)=\vec q_i^r$ and
generate reference momenta $\vec p_i(0)$ from the
conditional Gibbs distribution, so that
$\vec p(0)\sim \Nrm(\vec 0, C)$, for
$C^{-1}= \beta \mathcal{G}(\vec q^r)\otimes I_2$. Then,
approximate $\vec p_i(t_n), \vec q_i(t_n)$ by
\cref{eq:euler}. We can now apply the explicit Euler method
to \cref{eq:a,eq:a1} to define a mapping $\vec\Phi\colon B\to B$.  It can be
shown \citep{Mills2008-ec} that the approximate $\vec\Phi$
is also a diffeomorphism when $\tstep$ is sufficiently
small. We show samples of the action of $\vec\Phi$ on a  rectangular grid in \cref{fig:pf} for different values of
the inverse temperature $\beta$.



\begin{figure}
	\centering
	  \raisebox{3cm}{$\beta=10$}%
	 \includegraphics{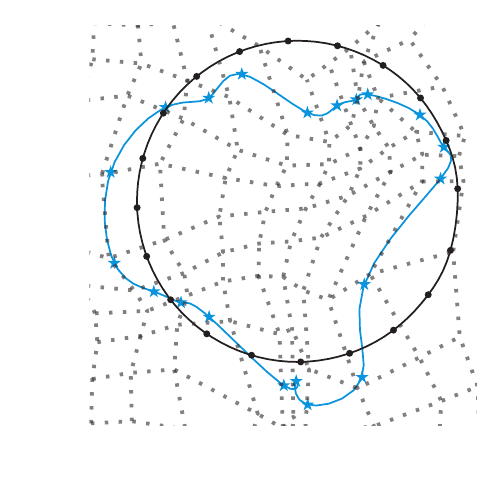}
	\raisebox{3cm}{ $\beta=20$}%
	\includegraphics{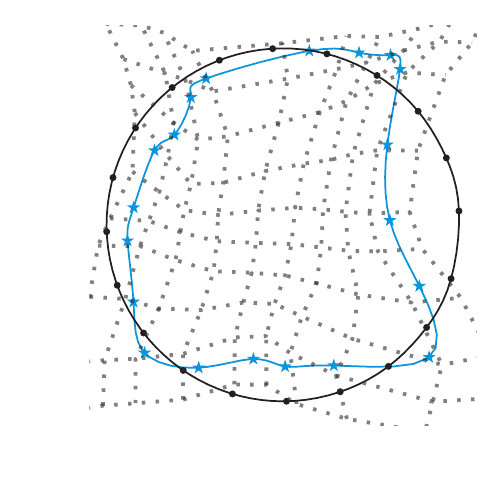}\\
\raisebox{3cm}{ $\beta=40$}%
 \includegraphics{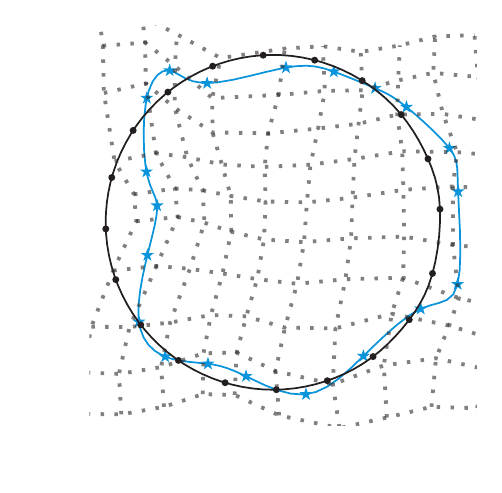}%
 \raisebox{3cm}{$\beta=80$}%
  \includegraphics{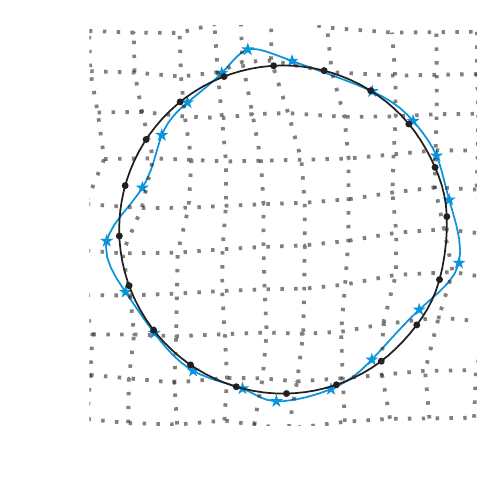}
  \caption{Push-forward maps $\vec\Phi$ applied to a grid on
    $B=[-1,1]^2$  and the unit circle (shown in blue) with $G(\vec q_1,\vec q_2)=\exp(-\Vpair{\vec q_1-\vec q_2}^2/\ell^2)$ for $\ell=0.5$, $\lambda=0.5$ and $\beta=10,20,40,80$, based on $\vec q_i^r$ as
    the marked points ({\color{blue} $\star$}) and $\vec p(0)\sim \Nrm(\vec 0, C)$ for $C^{-1}=\beta \mathcal{G}(\vec q^r)\otimes I_2$. As the
    inverse temperature $\beta$ is increased, the circle is
    pushed forward to smoother shapes. }\label{fig:pf}
\end{figure}
\section{Approximation of generalised Langevin equations}\label{sec:appr_lang}


Suppose that reference and target landmarks $\vec q^r_i$ and
$\vec q^t_i$ are known exactly. In
Bayesian statistics, the prior distribution is conditioned
on the data (landmarks in our case) to define a posterior
distribution (on the paths $\vec p(t), \vec q(t)$,
and hence on diffeomorphisms $\vec\Phi$).  For the generalised Langevin
prior with Gibbs initial data and exact landmark data, the posterior distribution on
$[\vec p(t), \vec q(t)]$ is generated by taking
solutions of \cref{eq:11} with initial data
$\vec q(0)=\vec q^r$ and
$\vec p(0)\sim \exp(-\beta H(\vec q^r,\cdot))$ and
conditioning on $\vec q(1)=\vec q^t$.
This is a type of diffusion bridge, which is important in parameter-estimation algorithms for SDEs; see \citep{Papaspiliopoulos2012-nl,Bladt2015-mm,MR2807970}.

In our case, the SDE gives a hypoelliptic diffusion and we
condition only on the position variables.  The problem is
similar to \citep{MR2807970}, which develops a stochastic
PDE for sampling Langevin diffusion bridges with the separable Hamiltonian $H=\frac 12 p^2+V(q)$ for
a potential $V$. It is not clear how their approach
generalises to the present situation with a non-separable  $H$. The method of analysis uses the Girsanov
theorem to replace \cref{eq:11} by a diffusion bridge for a linear SDE
\citep{Delyon2006}. The linear SDE has a Gaussian
distribution and standard formulas for conditioning Gaussian
distributions are available. This technique underlies
several approaches to sampling diffusion bridges such as
\citep{Golightly2008-zs,MR2807970}. In the present
situation, Girsanov is much less effective, as the
nonlinearities in the position equation due to
$\nabla_{\vec p_i} H=\sum_{j=1}^N\vec p_i G(\vec q_i, \vec
q_j)$ are unchanged by Girsanov's transformation and it is
hard to find a linear SDE to work with.

Other approaches to sampling diffusion bridges include
\citep{Bladt2015-mm}, which is not developed in the hypoelliptic
case, or the Doob h-transform
\citep{Papaspiliopoulos2012-nl}, which is computationally
very demanding, as it involves computing the full pdf of the
diffusion. Unfortunately, none of the known methods for diffusion bridges works with \cref{eq:11} to give computationally convenient algorithms.

Without an efficient method for sampling the
diffusion bridge, it is hard to formulate a Monte
Carlo Markov Chain method with good acceptance
rates. Consequently, the generalised Langevin prior
distribution is difficult to use in Bayesian statistics and
we now turn to simpler prior distributions, which arise by approximating the
Langevin equation. We introduce three priors, one based on a linearised Langevin equation and two based on the Baker--Campbell--Hausdorff formula for operator splittings.

All three of these methods are based on a regime of small dissipation $\lambda$ and large inverse temperature $\beta$. In this case, 
sample paths of the Langevin equation are close
to those of the Hamiltonian system on the time interval $[0,1]$. This is a reasonable
assumption in applications, as we want the time scale
$1/\lambda\gg 1$, so that the landmarks $\vec q_i(t)$ at
$t=0$ and $t=1$ are well-coupled, but there is some drift in them. As we saw in \cref{fig:pf}, 
small $\beta$ leads to large perturbations of the initial
shape. Therefore we assume that $\sigma^2=2\lambda/\beta$ is small for computational convenience.
In cases where these assumptions are not sufficient, it may be necessary to consider a higher-order method, but we do not do that here.

\subsection{Linearised Langevin equation}\label{sec:method}
In this section, based on small $\sigma^2$, we linearise the Langevin equation about the
Hamiltonian solution to define a Gaussian prior distribution.

Let $\hat{\vec z}(t)=[\hat{\vec p}(t),\hat{\vec q}(t)]$
denote a solution of \cref{eq:ham}.  Write the solution $\vec z(t)=[\vec p(t), \vec q(t)]$ of \cref{eq:11} as $\vec z(t)=\hat{\vec z}(t)+\vec{\delta}(t)+\vec {R}(t)$, where $\vec \delta(t)$ is a first-order correction given by linearising \cref{eq:11} around $\hat{\vec z}(t)$. With initial conditions $\vec{\delta}(t^*)=\vec z(t^*)-\hat{\vec z}(t^*)$, it is defined by the linear system of
SDEs
  \begin{equation}
    d\vec \delta%
    =    \bp{      -\lambda%
      \begin{pmatrix}
        \nabla_{\vec p} H(\hat{\vec z}(t))\\%
        \vec 0
      \end{pmatrix}%
      +B^+(t)\vec \delta    } \,dt%
    +    \begin{pmatrix}
      \sigma I_{dN} \\ 0
    \end{pmatrix}d \vec W(t),%
    \label{eq:linear}
  \end{equation}
  where $\vec W(t)$ is a $\real^{dN}$ Brownian motion and
  \[
  B^+(t)=
  \begin{pmatrix}
    -\lambda \nabla_{\vec p \vec p} H - \nabla_{\vec q \vec
      p} H%
    & -\lambda \nabla_{\vec p \vec q} H - \nabla_{\vec q \vec q} H\\
    \nabla_{\vec p\vec p} H%
    & \nabla_{\vec p\vec q} H
  \end{pmatrix},
  \]
  all evaluated at $\hat{\vec z}(t)$. In the case
  $\lambda=\sigma=0$, $\vec\delta=\vec 0$
  solves~\cref{eq:linear}.  With smoothness and growth conditions on $H$, it can be shown that the remainder $\vec R(t)=\order{\sigma^2+\lambda^2}$ \cite{Freidlin2012-dw}.

  To preserve the symmetry of the
  system, we specify an initial distribution at $t^*=1/2$ and
  ask that $\vec \delta(t^*)\sim \mu^*$.  For
  $t<1/2$, we use
  \begin{equation}
    d\vec \delta%
    =    \bp{      \lambda%
      \begin{pmatrix}
        \nabla_{\vec p} H(\vec p^*(t), \vec q^*(t))\\%
        \vec 0
      \end{pmatrix}%
      +B^-(t)\vec \delta    } \,dt%
    +    \begin{pmatrix}
      \sigma I_{dN} \\ 0
    \end{pmatrix}d \vec W(t),%
  \end{equation}
  for
  \[
    B^-(t)%
    =
    \begin{pmatrix}
      \lambda \nabla_{\vec p \vec p} H - \nabla_{\vec q \vec
        p} H%
      & \lambda \nabla_{\vec p \vec q} H - \nabla_{\vec q \vec q} H\\
      \nabla_{\vec p\vec p} H%
      & \nabla_{\vec p\vec q} H
    \end{pmatrix}.
  \]
  That is, the sign of the dissipation is switched as we are
  specifying a final condition for this system. $B^-$ differs
  by a sign in the conservative terms, as time is reversed.

  Equation \cref{eq:linear} is linear, its solution is a Gaussian process, and exact expressions are
  available for the mean and covariance in
  terms of deterministic integrals \citep{karatzas1998brownian}. We prefer to use a
  time-stepping method to approximate \cref{eq:linear}. We
  specify the distribution at some intermediate time, and
  need forward and backward integrators: The Euler--Maruyama
  method gives approximations $\vec \delta_n\approx \vec \delta(t_n)$ defined by
  \[
    \vec \delta_{n+1} = \underbrace{\pp{ I+ B^+_n \tstep
      }}_{\eqcolon M^+_n}\vec \delta_n%
    +\vec A_n  %
    +\begin{pmatrix}
      \sigma \Delta \vec W_n\\
      \vec 0
    \end{pmatrix},\quad \text{use for $t_{n+1}>1/2$,}
  \]
  \[
    \vec \delta_{n-1}%
    = \underbrace{\pp{ I+ B^-_n \tstep } }_{\eqcolon M^-_n}\vec
    \delta_n%
    + \vec A_n  %
    + \begin{pmatrix}
      \sigma \Delta \vec W_n\\
      \vec 0
    \end{pmatrix},\quad \text{use for $t_{n-1}<1/2$,}
\]
where
\[
\vec A_n=-\tstep\,\lambda
\begin{pmatrix}
\nabla_{\vec p} H\\ \vec 0
\end{pmatrix}
\]
\[
B^+_n=B(t_n),\qquad%
B^-_n%
=-B^-(t_n)%
=\begin{pmatrix} -\lambda \nabla_{\vec p\vec p} H +
  \nabla_{\vec q \vec p} H %
  & -\lambda \nabla_{\vec p \vec q} H + \nabla_{\vec q \vec q} H\\
  -\nabla_{\vec p\vec p} H%
  & -\nabla_{\vec p\vec q} H.
\end{pmatrix}
\]
For a Gaussian
initial distribution $\mu^*$, the resulting distribution on
paths and their Euler--Maruyama approximation are Gaussian. In
\cref{sec:lin_dist}, we give equations for calculating the
mean and covariance of the Euler--Maruyama approximations
$[\vec \delta_0,\dots,\vec \delta_{N_\tstep}]$.




The Gaussian distributions can be sampled to generate paths
$[\vec p(t),\vec q(t)] \approx [\hat{\vec p}(t), \hat{\vec q}(t)]+\vec \delta(t)$. This then
defines a map $\vec \Phi$ via \cref{eq:a,eq:a1}. Note however
that the consistency is broken and $\vec{\Phi}(\vec q_i(0))$ may
not equal $\vec q_i(1)$.
\subsection{Operator splitting}

Let $L$ denote the generator associated to the generalised Langevin equation \eqref{eq:11}. Then,
$L=L_0+ \sigma^2 L_1$ for
\begin{align*}
  L_0%
  &= \nabla_{\vec p} H \nabla_{\vec q}%
  -\nabla_{\vec q}H \nabla_{\vec p},\qquad%
  \text{known as the Liouville operator, and}\\
  L_1%
  &= \frac{1}{\sigma^2}\pp{-\lambda \nabla_{\vec p} H
    \nabla_{\vec p}%
    +\frac 12 \sigma^2 \nabla_{\vec p}\cdot \nabla_{\vec
      p}}%
  ={-\frac{\beta}{2} \nabla_{\vec p} H \nabla_{\vec p}%
    +\frac 12 \nabla_{\vec p}\cdot \nabla_{\vec p}}.
\end{align*}
The Fokker--Planck equation is $\rho_t=L^* \rho$, where
$L^*$ denotes the adjoint of $L$, and
describes the evolution of the pdf from a
given initial density $\rho(0,\cdot)=\rho_0$. Using semigroup
theory, we write $\rho(t,\cdot)=e^{L^* t}\rho_0$.
  We
can approximate $e^{A+B}$ via $e^{A} e^{B}+\order{[A,B]}$ or
via the Strang splitting  as
  \[
  e^{A+B}\approx e^{A/2} e^B e^{A/2} + \order{[B,[B,A]]+[A,[A,B]]},
  \]
  where $[\cdot,\cdot]$ denotes the operator commutator.
  This can be applied with $A=L^*_0$ and $B=\sigma^2L^*_1$ to simplify \cref{eq:11}.
   In the small-noise limit,
  $\sigma^2 L_1^*\to 0$, but $L_0$ is
  order one and the error is $\order{\sigma^2}$. These approximation
  strategies do preserve the Gibbs invariant measure, as
  $e^{\sigma^2 L_1^*}\mu=e^{L_0}\mu=0$
  for $\mu=\exp(-\beta H)$. They are also much easier to
  compute with than the full
  $e^{L^*}$. We look at two uses of the Strang splitting:
  \begin{description}

\item[First splitting]  Approximate
\[
e^{L^*}\approx
e^{\sigma^2L_1^*/2}%
e^{L_0^*}%
e^{\sigma^2L_1^*/2}.
\]
The semigroup on the right-hand side maps
\[
[\vec p(0),\vec q(0)]%
\underbrace{\longmapsto}_{e^{\sigma^2L_1^*/2}} %
[\vec p(1/2),\vec  q(0)]%
\underbrace{\longmapsto}_{e^{L_0^*}} %
[\tilde {\vec p}({1/2}),{ \vec{q}}(1)]%
\underbrace{\longmapsto}_{e^{\sigma^2L_1^*/2}} %
[\vec {p}(1),\vec q(1)].
\]

The two steps with $e^{\sigma^2 L_1^*/2}$ are described by the time-half evolution governed by the Ornstein--Uhlenbeck SDE
\begin{equation}
d\vec p%
=-\lambda \nabla_{\vec p} H(\vec p,\vec q_0)\,dt%
+\sigma\, d\vec W(t),\qquad %
\vec p(0)=\vec p_0,%
\label{genou}
\end{equation}
for  $[\vec p_0,\vec q_0]=[\vec p(0), \vec q(0)]$ or $[\tilde{\vec p}(1/2), \vec q(1)]$. This only involves a change in momenta. The middle step with $e^{ L_0^*}$ is the time-one evolution with the Hamiltonian equations \cref{eq:ham}.
If $[\vec {p}(0), \vec q(0)]\sim \exp(-\beta H)$, then
so are $[\vec {p}(1/2), \vec q(0)]$,
$[\tilde{\vec p}(1/2), \vec q(1)]$, and also
$[\vec p(1),\vec q(1)]$. The effects of $e^{\sigma^2 L_0^*/2}$ at
either end are superfluous, as they change the
momentum only; any conditioning is applied on the position
data. In this way, we see fit to disregard this term and
define the prior as the push forward under the Hamiltonian
flow of Gibbs' distribution. The density of the prior on
paths $\vec z(t)=[\vec p(t),\vec q(t)]$ for $t\in [0,1]$ is
  \[
  \exp(-\beta H(\vec z(0)))%
  \delta_{\vec z(t)-\vec{S}{(t; 0,\vec z(0))}},
  \]
  where $\vec S(t; s,\vec z_0)$ is the solution of
  \cref{eq:ham} at time $t$ with initial data
  $[\vec p(s),\vec q(s)]=\vec z_0$.
  \item[Second splitting] Approximate
    \[
    e^{L^*}\approx
    e^{L_0^*/2}e^{\sigma^2L_1^*}e^{L_0^*/2}.
    \]

The semigroup on the right-hand side maps
\[
  [\vec p(0),\vec q(0)]%
  \underbrace{\longmapsto}_{e^{L_0^*/2}} %
  [(\vec p(1/2),
  \vec q(1/2)]%
  \underbrace{\longmapsto}_{e^{\sigma^2L_1^*}} %
  [(\tilde{
    \vec{p}}(1/2), \vec q(1/2)]%
  \underbrace{\longmapsto}_{e^{L_0^*/2}} %
  [\vec p(1), \vec
  q(1)].%
\]
Again, if $[\vec p(0),\vec q(0)]\sim \exp(-\beta H)$, then so do each of the following sets of positions and momenta. It is important to preserve each of the three parts of the approximation, as the Hamiltonian flow at either end affects all components.
The density is
\[
\exp(-\beta H(\vec p(1/2), \vec q(1/2))\,%
\upsilon(1, \tilde{\vec{p}}(1/2); [\vec p(1/2),\vec q(1/2)]) \,%
\delta_{\vec z(t)-\vec Z(t)}
\]
where
$\upsilon(t, \vec p; [\vec p_0,\vec q_0])$ is the density at time $t$ of the random variable $\vec p(t)$ defined by the SDE
\begin{equation}
d\vec p%
=-\lambda \nabla_{\vec p} H(\vec p,\vec q_0)\,dt%
+\sigma d\vec W(t),\qquad \vec p(0)=\vec p_0.\label{eq:ou}
\end{equation}
The function $\vec Z(t)$ describes the Hamiltonian flow and is defined by
\begin{equation}\label{Z}
\vec Z(t)=
\begin{cases}
S(t; 1/2, [\tilde{\vec {p}}(1/2),\vec q(1/2)]),& t>1/2;\\
S(t; 1/2, [{\vec {p}}(1/2),\vec q(1/2)]),& t<1/2.
\end{cases}
\end{equation}
It will be more convenient to have both halves flow forward and write
\[
\vec Z(t)=
\begin{cases}
S(t-1/2; 0, [\tilde{\vec {p}}(1/2),\vec q(1/2)]),& t>1/2;\\
R S(1/2-t; 0, R [{\vec {p}}(1/2),\vec q(1/2)]),& t<1/2,
\end{cases}
\]
where $R[\vec p, \vec q]=[-\vec p, \vec q]$ expresses the time reversal.

The key variables for conditioning are the start and
end positions, $\vec{q}(0)$ and $\vec{q}(1)$. These positions are deterministic maps of the
time-half data, provided by a time-half push forward
of the deterministic Hamiltonian dynamics. Thus, it is convenient to
express the prior in terms of $\vec p(1/2), \vec q(1/2),
\tilde{\vec p}(1/2)$ by the density proportional to
 \[
 \exp(-\beta H(\vec p(1/2), \vec q(1/2))\,%
 \upsilon(1/2, \tilde{\vec{p}}(1/2); [\vec p(1/2), \vec q(1/2)]).
\]

We now show how to simplify $\upsilon$ when $\beta$ is large and $\lambda$ is small.
In \cref{eq:ou}, $\nabla_{\vec p} H(\vec p,\vec q) =(\mathcal{G} (\vec q) \otimes I_d ) \vec p$ for  $\mathcal{G}(\vec q)$ defined in \cref{notation}.  For a deterministic $\vec p_0$, the solution $\vec p(t)$ of \cref{eq:ou} is an Ornstein--Uhlenbeck process with  a Gaussian distribution with
mean
$\mu_t =(e^{-\lambda \mathcal{G}(\vec q_0)t }\otimes I_d)\,\vec
p_0$ and covariance $C_t\otimes I_d$, for
\[
C_t%
\coloneq \sigma^2\frac1{2\lambda}\,\mathcal{G}(\vec q_0)^{-1}%
\pp{I_N-e^{-2 \lambda \,t\, \mathcal{G}(\vec q_0)}}%
=\frac{1}{\beta}\mathcal{G}(\vec q_0)^{-1}%
\pp{I_{N}-e^{-2 \lambda \,t\,\mathcal{G}(\vec q_0)}}.
\]
By Taylor's theorem, $e^{-\lambda A}=I_N-\lambda A +\int_0^1 \lambda^2 A^2 e^{-\lambda A s} (1-s)\,ds$ for any $N\times N$ matrix $A$. Hence,
\begin{align*}
C_t%
&=\sigma^2 \,t\,I_N+\frac{1}{\beta }\,%
 \mathcal{G}(\vec q_0)^{-1}%
 \int_0^1 4 \,\lambda^2 \,t^2\, \mathcal{G}(\vec q_0)^2 e^{-2\, \lambda\, t\, \mathcal{G}(\vec q_0)} (1-s)\,ds\\
 &=\sigma^2 \,t\,I_N + 4 \frac{1}{\beta} \,\lambda \,t\, K,\qquad\text{ for }%
  K%
  \coloneq \int_0^1 \lambda\,t\,\mathcal{G}(\vec q_0) \,%
  e^{-2 \,\lambda \,t\,\mathcal{G}(\vec q_0)}%
   (1-s)\,ds.
\end{align*}
When $G$ is a positive-definite function, $K$ is uniformly bounded over any $\vec q_0\in \real^{dN}$ and $t\in[0,1]$. Therefore,
\begin{equation}
C_t=\sigma^2\,t\, I_N+\order{\lambda\, t/\beta}.\label{eq:Ct}
\end{equation}
As explained in \cref{sec:appr_lang}, we are interested in large $\beta$ and small $\lambda$ and hence we are justified in approximating $C_t\approx \sigma^2\,t\, I_N$ for $t\in [0,1]$.
Then,
\[
e^{\sigma^2 L_1^* t} \delta_{(\vec p_0,\vec q_0)}
\approx \Nrm((e^{-\lambda \, t\,
	\mathcal{G}(\vec q_0)} \otimes I_d)\vec p_0 , \sigma^2\,t\, I_{dN}) \times \delta_{\vec q_0}.
\]
For the prior, we are interested in $\upsilon(1, \tilde{\vec p}(1/2);(\vec p(1/2), \vec q(1/2)))$ and, by this approximation,
\[
\upsilon(1, \cdot; (\vec p(1/2), \vec q(1/2)))
\approx \Nrm( (e^{-\lambda \,    \mathcal{G}(\vec q(1/2) ) } \otimes I_d)\vec p(1/2), \sigma^2 I_{dN}) \times \delta_{\vec q(1/2)}.
\]
Hence, the prior distribution on $(\vec
p(1/2), \vec q(1/2), \tilde{\vec p}(1/2))$ has density proportional to
\begin{equation}
 \exp\pp{-\beta H(\vec p(1/2),\vec q(1/2))}%
\exp\pp{-\frac 1{2\sigma^2}\norm{\tilde{\vec p}(1/2)%
		-(e^{-\lambda \,    \mathcal{G}(\vec q(1/2) ) } \otimes I_d)%
    \vec p(1/2)}^2}.\label{prior2}%
\end{equation}
Distributions on the paths $[\vec p(t),\vec q(t)]$ are
implied by solving \eqref{eq:ham} with initial data
$[\vec p(1/2)$, $\vec q(1/2)]$ for $t>1/2$ and with final
data  $[\tilde{\vec p}(1/2), \vec{q}(1/2)]$ for $t<1/2$.
\end{description}

\section{Data and experiments}\label{sec:bayes}

We now show how to work with the prior distributions using
data.  For a prior distribution on the diffeomorphisms
$\vec\Phi$, we would like to compute or sample from the
conditional distribution of $\vec\Phi$ given that
$\vec q_i(0)=\vec q_i^t+ \vec \eta_i^t$ and
$\vec q_i(1)=\vec q_i^r+\vec \eta_i^r$, where
$\vec \eta_i^t, \vec \eta_i^r\sim \Nrm(\vec 0, \delta^2 I_d)$
\iid for some parameter $\delta>0$.  We present three cases:
\begin{enumerate}
\item The linearised-Langevin prior is Gaussian and
  conditioning by observations of the landmarks with \iid
  Gaussian errors yields a Gaussian posterior
  distribution. We show how to compute the posterior
  distribution for the Euler--Maruyama discretised equations.
\item The first splitting prior consists of a Gibbs
  distribution on the initial data and Hamiltonian flow
  equations. As such the distribution is specified by the
  distribution on the initial landmarks and generalised
  momenta. We condition this on landmarks also with \iid
  Gaussian errors. The posterior is not Gaussian. We show
  how to compute the MAP point and approximate the posterior
  covariance matrix by the Laplace method. The MAP point is a set
  of initial landmark positions and generalised momenta.
\item The second splitting prior consists of a Gibbs
  distribution on the midpoint, a second momenta (correlated
  to the first) at the midpoint, and Hamiltonian flow
  equations. This distribution is parameterised by one set
  of landmarks and two sets of generalised momenta. We show
  how to examine the posterior distribution (again
  conditioning on Gaussian observations) via the MAP point and Laplace method. We interpret the
  MAP point as an average set of landmarks, by extending the prior to allow for multiple sets of landmarks.
\end{enumerate}

The discussion includes computational examples. The
calculations were performed in Python using the Numpy,
Matplotlib, and Scipy libraries and the code is available for download \cite{reg_sde}. For information about the code and for a set of further examples, see the Supplementary Material. In all cases, the landmarks in each image were centred to have zero
mean and then aligned using an orthogonal Procrustes transformation in order
to remove potentially confusing global transformations.

\subsection{Noisy landmarks via the linearised-Langevin equation}\label{ex_lin_lan}

The key step in defining the linearised-Langevin prior is distinguishing paths about which to linearise. We choose paths $[\vec p(t), \vec q(t)]$ by solving the Hamiltonian boundary-value problem \cref{eq:ham} based on the landmark data $\vec q_i^t$ and $\vec q_i^r$. Then, the linearised-Langevin prior is a Gaussian distribution on
the paths $[\vec p(t),\vec q(t)]$ generated by
\eqref{eq:linear}, the linearisation of the Langevin equations about the distinguished paths. We denote the Euler--Maruyama
approximation with time step $\tstep=1/N_{\tstep}$ to
$[\vec p^*,\vec q^*]+\vec \delta$ at $t_n$ by
$[\vec P_n,\vec Q_n]$ and the vector
$[\vec P_0,\vec Q_0,\dots, \vec P_{N_\tstep},\vec
Q_{N_\tstep}]$ by $\vec X$. The mean $\vec M_1$ and
covariance $\mathcal{C}$ of $\vec X$ can be found using the
equations in \cref{sec:lin_dist}.


Let $\widehat{\vec Q}^r=\vec Q_0+\vec \eta^r$ and
$\widehat{\vec Q}^t=\vec Q_{N_\tstep}+\vec \eta^t$ for
$\vec \eta^r,\vec \eta^t\sim \Nrm(\vec 0,\delta^2 I_{dN})$ \iid (the
distributions are independent of each other and also of the
Brownian motions). Let
$\vec Y=[\widehat{\vec Q}^r,
\widehat{\vec Q}^t]$ and $\vec Z=[\vec X, \vec Y]$.
$\vec Z$ is then Gaussian with mean
$
\bp{\vec M_1,\vec M_2}=
  \bp{\vec M_1,
    \bp{\mean{\vec Q_0},\mean{\vec Q_{N_\tstep}} }}%
$
and covariance
    \[
    \begin{bmatrix}
      C_{11} & C_{21}^\trans\\
      C_{21} & C_{22}
    \end{bmatrix},\quad
\text{ where  $C_{11}=\mathcal{C}$, }
C_{22}=
\begin{bmatrix}
\cov(\vec Q_0, \vec Q_0)+\delta^2 I_{d N}&C_{0N_{\tstep}}\\
C_{0N_{\tstep}}^\trans&  \cov(\vec Q_{N_{\tstep}}, \vec
Q_{N_{\tstep}})+\delta^2 I_{d N}
\end{bmatrix},
\]
\[\text{ and } C_{21}=
  \begin{bmatrix}
  \cov(\vec Q_0, \vec Q_0)& \dots&\cov(\vec Q_0, \vec
Q_{N_\tstep})\\
  \cov(\vec Q_{N_\tstep}, \vec Q_0)& \dots&\cov(\vec Q_{N_\tstep}, \vec Q_{N_\tstep})
\end{bmatrix}.
\]
The distribution of $\vec X$ given observations
$\widehat{\vec Q}^t=\vec q^t$ and $\widehat{\vec Q}^r=\vec q^r$
is $\Nrm(\vec M_{1|2}, C_{1|2})$ with
\begin{align*}
  \vec M_{1|2}&= \vec M_1 %
  + C_{12} C_{22}^{-1}(\vec y-\vec  M_2),\qquad \vec y=[\vec
  q^r,\vec q^t],\\
  C_{1|2}&= C_{11}- C_{12} C_{22}^{-1} C_{21}.
\end{align*}
For the number of landmarks that we consider %
(less than a hundred), this is readily computed using
standard linear-algebra routines. The two inverse matrices
involved are of size $dN\times dN$. The full covariance
matrix is memory demanding though, as it has size $(N_\tstep+1)
2dN \times (N_\tstep+1) 2dN$.

\cref{fig:lin1} shows the solution of \cref{eq:ham} and the associated registration for a set of known landmarks. We linearise about the solution $[\vec p(t), \vec q(t)]$, to define a linearised-Langevin prior and \cref{fig:lin} shows the standard deviations of the computed posterior distribution at the landmark positions. \cref{fig:lin_var} shows the standard deviation of the posterior throughout the image space, in both the original and warped co-ordinate systems. The difference in standard deviations shown in \cref{fig:lin_var,fig:lin} is significant, as one comes from the posterior distribution matrix at the landmarks and the other by a Monte Carlo estimator of the distribution of $\vec \Phi(\vec Q)$ for $\vec Q$ away from landmark points. In this linearised situation, $\vec \Phi$ may not agree with the linearised Langevin equation. We see  this weakness again for large deformations in \cref{fig:lin_consist}, where we compare the random diffeomorphisms and the paths $\vec q_i(t)$ defined by samples of the posterior distribution. Though $\vec \Phi(\vec q_i^r)$ and $\vec q_i(1)$ agree when the data is regular, for larger deformations, there is significant disagreement. This is because $\vec \Phi$ is defined by \cref{eq:a,eq:a1}, which is no longer identical to the linear equation \cref{eq:linear} used to define $\vec q_i(t)$.

\begin{figure}
  \centering
  \includegraphics{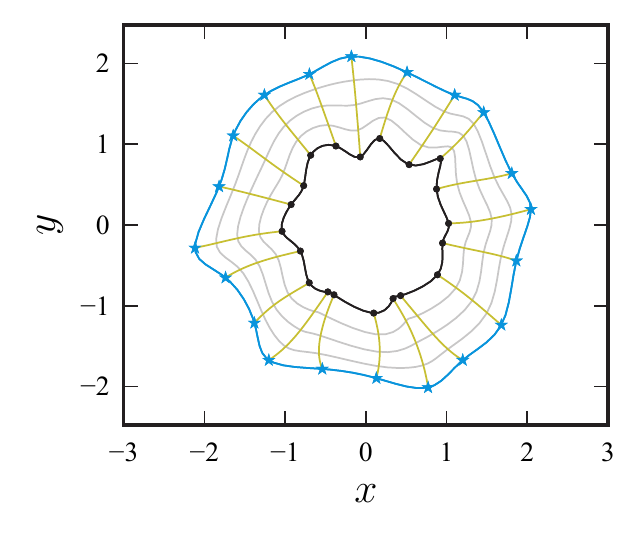}
  \caption{The blue and black stars mark twenty noisy observations of regularly spaced points on two concentric circles. Using the Hamiltonian boundary-value problem \cref{eq:ham}, we compute a diffeomorphism and show paths $\vec q_i(t)$. Three intermediate shapes are shown in grey. The yellow lines show the paths taken by the landmarks through the interpolating shapes.}\label{fig:lin1}
\end{figure}
\begin{figure}
  \centering
    \includegraphics{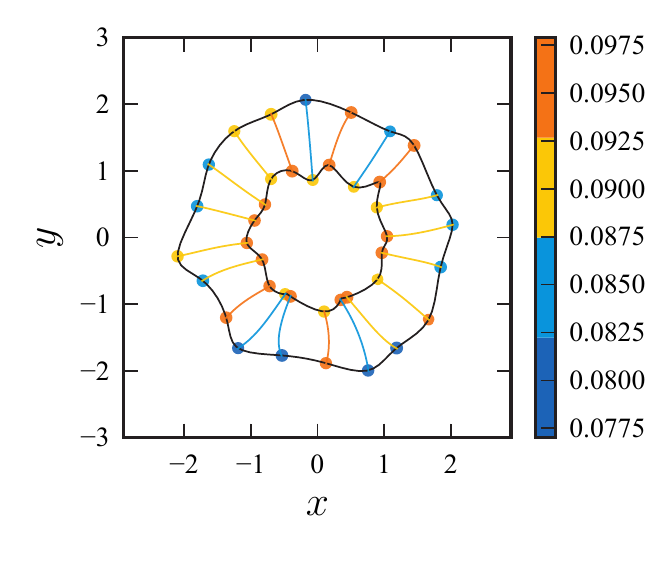}
    \includegraphics{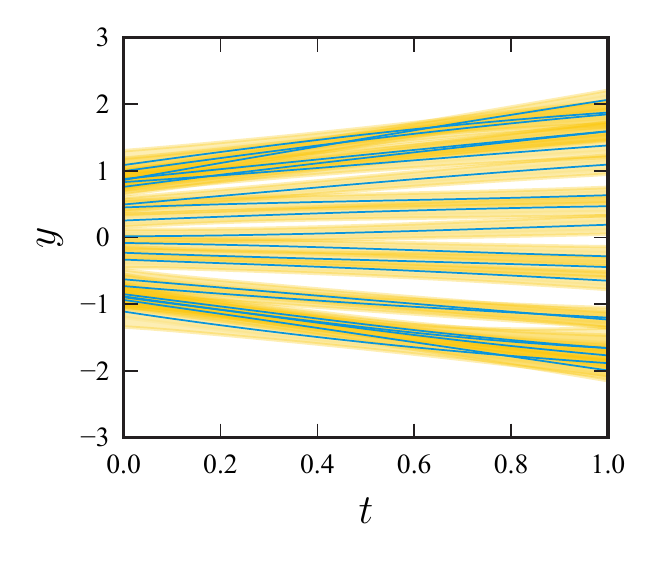}
  \caption{A registration between two noisy observations of a circle at different scales (radii of 1 and 2 correspondings to times $t=0, 1$ respectively) using the linearised-Langevin
    prior (with $\lambda=0.1$ and $\beta=25$), with landmarks observed with \iid
    $\Nrm(\vec 0, \delta^2 I_d)$ errors for $\delta^2=0.01$. The discs on the left-hand plot
    and the yellow shadows on the right-hand plot indicate one
    standard deviation of the computed posterior covariance
    matrix.}\label{fig:lin}
\end{figure}
\begin{figure}
	\centering
	\includegraphics{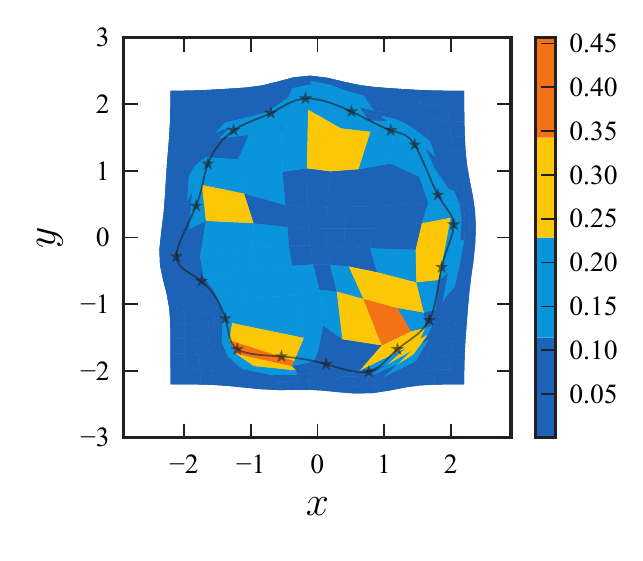}
	\includegraphics{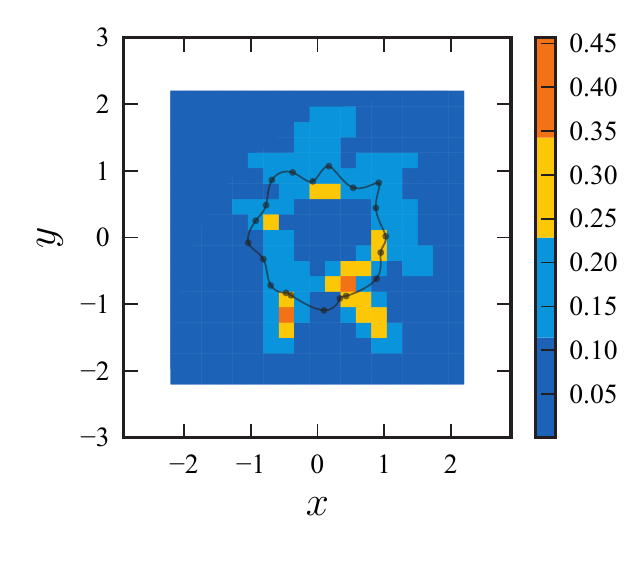}
	\caption{The colours shows the standard deviation of $\vec\Phi(\vec Q)$ at $\vec\Phi(\vec Q)$ (left-hand side) and at $\vec Q$ (right-hand side), for a set of uniformly spaced $\vec Q$ on a rectangular grid, when $\vec\Phi$ is defined by the posterior distribution for the linearised-Langevin prior. }\label{fig:lin_var}
\end{figure}
\begin{figure}
	\centering
	\includegraphics{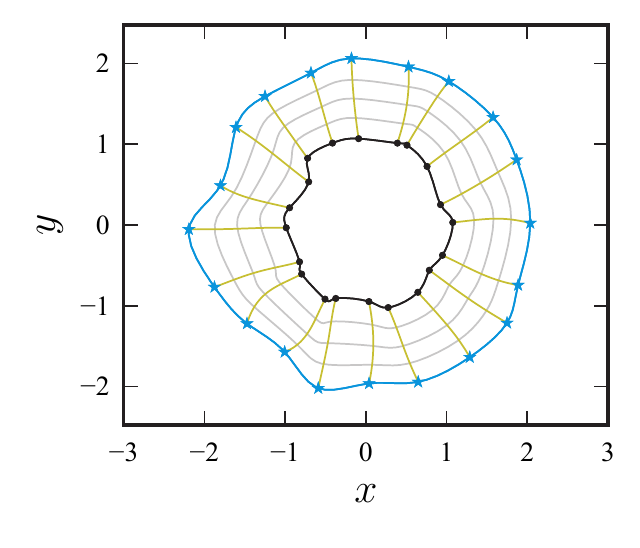}%
	\includegraphics{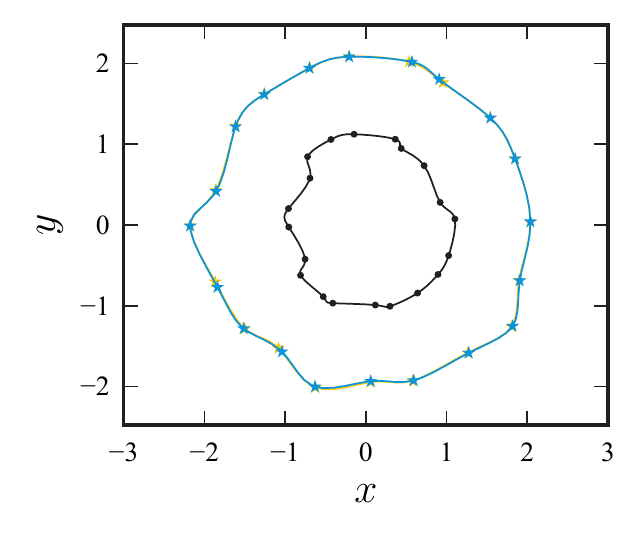}\\%
	\includegraphics{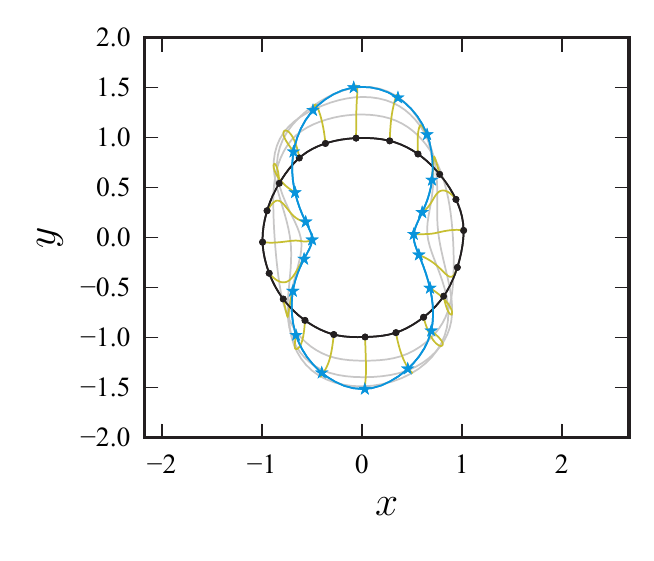}%
	\includegraphics{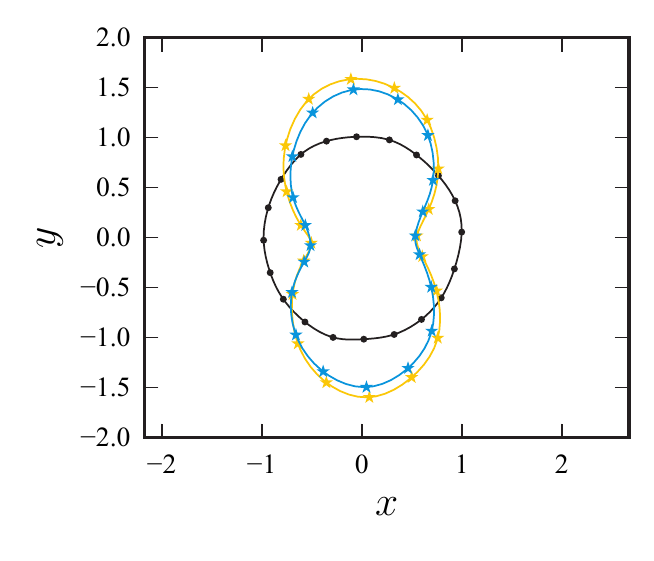}%
	\caption{In the right-hand column,  blue stars mark $\vec\Phi(\vec q_i(0))$ and the yellow stars marks $\vec q_i(1)$, where $\vec q_i(t)$ and $\vec\Phi$ are given via samples of the linearised-Langevin posterior distribution, with paths shown in the left-hand column. The inner black loop marks $\vec q_i(0)$. Due to the linearisation, $\vec \Phi(\vec q_i(0))\ne \vec q_i(1)$, although it is  closer on the top row where the deformation field is much smoother.}%
	\label{fig:lin_consist}
\end{figure}

\subsection{Noisy landmarks by operator splitting}\label{ss:nlos}

The first splitting prior is much less memory demanding than
the linearised-Langevin prior, as the randomness concerns
only the initial position and momenta. It also has the
advantage of preserving the Gibbs distribution and maintaining
consistency with the definition of $\vec\Phi$.
We show how to use this prior in the same scenario as
\cref{ex_lin_lan}. This time we are unable to sample the posterior distribution. Instead, we formulate a MAP
estimator and apply a Laplace approximation to estimate the
posterior covariance.

As the observation error is independent of the prior, the
posterior density is given by the pdf of the prior on
$[\vec p_0, \vec q_0]$ times the data likelihood for the
observations $\vec q_0$ and $S_q(1,0;[\vec p_0,\vec q_0])$ of $\vec q^r$ and $\vec q^t$. The density of the first
splitting prior is $ \exp(-\beta H(\vec p_0,\vec q_0))$. For
Gaussian observations, the data likelihood is proportional
to
\[
 \exp\pp{ -\frac{1}{2\delta^2}\pp{\norm{\vec q^r-\vec q_0}^2%
    +\norm{\vec q^t-S_q(1; 0, [\vec p_0,\vec q_0
    	])}^2}},
\]
where $S_q$ denotes the position components in $S$ (the Hamiltonian flow map).
The posterior density is proportional to
\[\exp(-\beta H(\vec p_0,\vec q_0))%
\exp\pp{-\frac{1}{2\delta^2}\pp{\norm{\vec q^r-\vec q_0}^2 +
    \norm{\vec q^t-S_q(1; 0, [\vec p_0,\vec q_0])}^2}}.
\]
To find the MAP point, we minimise
\begin{equation}\label{MAP1F}
F(\vec p_0,\vec q_0)\coloneq  \beta H(\vec p_0,\vec q_0)%
  +\frac{1}{2\delta^2}%
  \pp{\norm{\vec q^r-\vec q_0}^2%
    +\norm{\vec q^t-S_q(1;0, [\vec p_0,\vec q_0])}^2}.
\end{equation}
This comprises the regulariser that comes from Gibbs' distribution
and two landmark-matching terms, and can also be derived as a Tychonov regularisation of the standard landmark registration problem. There is one parameter $\beta$ from the Gibbs distribution and the dissipation $\lambda$ is not present. We minimise $F$ to find an approximation to
the MAP point, using standard techniques from unconstrained
optimisation and finite-difference methods for \cref{eq:ham}.

The Laplace method gives an approximation to
the posterior covariance matrix by a second-order approximation to $F$ at
the MAP point $\vec z_0$. Thus we evaluate the Hessian $\nabla^2 F$ of $F$
at the MAP point. Second derivatives of $F$ are approximated
by using a Gauss--Newton approximation for the last term, so we use
\[
\nabla^2 F%
\approx\beta \nabla^2 H %
+ \frac{1}{\delta^2}%
\bp{\begin{pmatrix}
 0 & 0 \\ 0 & I_{dN}
\end{pmatrix} + J^\trans J},
\]
where $J$ is the Jacobian matrix of $S_q(1;0,\vec z_0)$. The Gauss--Newton approximation
guarantees that the second term is positive definite (though
the Hessian of $H$ and the overall expression may not be).
To make sure the covariance is a well-defined  symmetric positive-definite matrix, we form a spectral decomposition of $\nabla^2 F$, throw away any
negative eigenvalues, and form the inverse matrix from the
remaining eigenvalues to define a covariance matrix
$C\approx \nabla^2 F^{-1}$.
See \cref{fig:split1} for an example.
\begin{figure}
  \centering
  \includegraphics{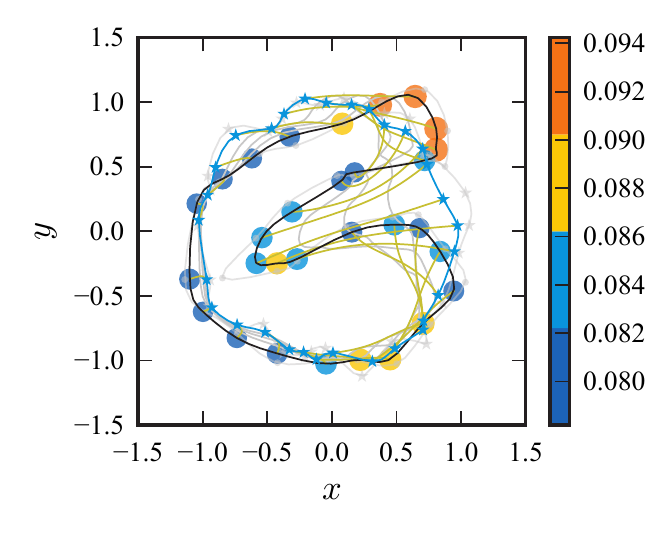}
  \includegraphics{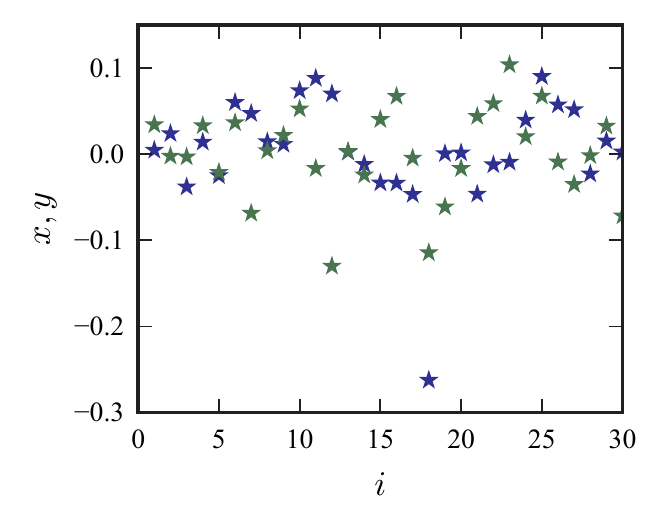}
  \caption{Noisy landmark registration (with $\delta=\sqrt{0.005}\approx 0.07$)
    using the first splitting prior (with
    $\beta=25$). In the left-hand plot, the grey lines mark the original landmark data. The landmarks given by the MAP registration algorithm are marked in colour, with discs indicating the
    standard deviation on the target landmarks by the Laplace approximation for the computed posterior
    covariance. The right-hand plots shows the difference between
    MAP landmarks and data landmarks.}\label{fig:split1}
\end{figure}

\subsection{Second splitting prior and landmark-set averages}
Averages are an important way of summarising a
data set. Under our Bayesian formulation, it is relatively simple to
define a consistent average for sets of landmarks defined on multiple images,
as we demonstrate in this section.  The approach is similar in spirit to the
arithmetic mean, which arises in calculations of
the MAP point for Gaussian samples.

We use the second splitting prior and
start with two sets of landmark points $\vec q^a$ and
$\vec q^b$. We wish to find a third set of landmark points
$\vec q^*$  that match both sets $a,b$ according to some measure.
We introduce momenta $\vec p^{*a}$ and
$\vec p^{*b}$. Classical landmark matching gives momenta $\vec
p^{*a}$ that flows $\vec q^*$ to $\vec q^{a}$, and
similarly for $b$. This can be done for any $\vec q^*$. The second splitting prior expresses our preference for less deformation and coupling of the two momenta, and makes $\vec q^*$ well defined.

The second splitting prior gives a
distribution on $(\vec p_{1/2}, \vec q_{1/2}, \tilde{\vec
  p}_{1/2})$ proportional to \cref{prior2}. Substituting $\beta=2\lambda/\sigma^2$, it is \begin{equation}
\exp\pp{-\beta H(\vec p(1/2),\vec q(1/2))}%
\exp\pp{-\frac \beta{4\lambda }\norm{\tilde{\vec p}(1/2)%
		-(e^{-\lambda \,    \mathcal{G}(\vec q(1/2) ) } \otimes I_d)%
		\vec p(1/2)}^2}.%
\end{equation}
 When coupled with the likelihood function for data $\vec q^r$ and $\vec q^t$ given by
\[
\exp\pp{-
\frac{1}{2\delta^2}%
\pp{%
	\norm{\vec q^r-S_q(1/2; 0, [-\vec  p_{1/2},\vec q_{1/2}] ) }^2%
	+\norm{\vec q^t-S_q(1/2; 0, [\tilde{\vec{p}}_{1/2},\vec{q}_{1/2}] ) }^2%
}%
},
\]
 we can write down the posterior pdf. Then, to find the MAP point, we minimise the objective function
\begin{gather}
\begin{split}
&F(\vec p_{1/2}, \vec q_{1/2},\tilde{\vec p}_{1/2})%
 \coloneq \beta H(\vec{p}_{1/2},\vec q_{1/2})%
+\frac {\beta}{4 \lambda}%
\norm{\tilde {\vec{p}}_{1/2}-e^{-\lambda
    \mathcal{G}(\vec{q}_{1/2})} \vec{p}_{1/2}}^2\\%
&\qquad  %
+\frac{1}{2\delta^2}%
\pp{%
	\norm{\vec q^r-S_q(1/2; 0, [-\vec  p_{1/2},\vec q_{1/2}] ) }^2%
    +\norm{\vec q^t-S_q(1/2; 0, [\tilde{\vec{p}}_{1/2},\vec{q}_{1/2} ]) }^2%
    }.
    \end{split}\label{eq:av2}
\end{gather}
This comprises the regulariser due to  Gibbs' distribution, a
penalty for changing the momentum at $t=1/2$, and two
landmark-matching terms. The minimiser of $F$ gives the MAP point. We are interested in using $\vec q^*=\vec q_{1/2}$ as the average landmark set.

Before discussing numerical experiments, we describe the limiting properties of the MAP point as $\lambda, \beta$ are varied. In the following, we assume that $B^N$ is a convex subset of $\real^{dN}$ and that $\vec q^r, \vec q^t\in B^N$.

\begin{lemma}\label[lemma]{lemma1}
	With $\vec p_{1/2}=\tilde{\vec p}_{1/2}=\vec 0$, the minimiser of
	\[
	f(\vec q)\coloneq{%
		\norm{\vec q^r-S_q(1/2; 0,[-\vec  p_{1/2},\vec q ] )}^2%
		+\norm{\vec q^t-S_q(0;1/2, [\tilde{\vec{p}}_{1/2},\vec{q}])  }^2%
	}
	\]
	over $\vec q_{1/2}\in B^N$
	is $\vec q_{1/2}=(\vec q^r+\vec q^t)/2$. Hence,
	\[
	\min_{(\vec p_{1/2}, \vec{q}_{1/2}, \tilde{\vec p}_{1/2}) \in \real^{dN}\times B^N \times \real^{dN}} F(\vec p_{1/2}, \vec{q}_{1/2}, \tilde{\vec p}_{1/2}) \le \frac{1}{4\delta^2} \norm{\vec q^r-\vec q^t}^2.
	\]
\end{lemma}
\begin{proof}
	 When $\vec p=\vec p_{1/2}=\tilde{\vec p}_{1/2}=\vec 0$,
	 $H=0$ and  $S_q(s;t,[\vec p,\vec q])=\vec q$ for all $s,t$.
	Hence, $f(\vec q)=\Vpair{ \vec q^r-\vec q}^2+\Vpair{\vec q^t-\vec q}^2$, which is minimised by $\vec q_{1/2}= (\vec q^r+\vec q^t)/2$.
\end{proof}

\begin{corollary} Assume that $G(\vec q_i)$ is uniformly bounded over $\vec q_i\in B\subset \real^d$. Then,
	as $\lambda\to 0$, $\vec q_{1/2}$ converges to $S_q(1/2;0,[\vec p_0,\vec q_0])$, where $[\vec p_0,\vec q_0]$ is the MAP point for \cref{MAP1F}.\label[corollary]{cor_lam}
\end{corollary}
\begin{proof} As $\min F$ is bounded independently of $\lambda$, we know that 
\[
\frac\beta\lambda\norm{\tilde{\vec p}_{1/2}-e^{-\lambda \mathcal{G} (\vec q_{1/2})} {\vec p}_{1/2}}^2
\]
is bounded as $\lambda\to 0$.  Hence, $\tilde{\vec p}_{1/2} - e^{-\lambda \mathcal{G}(\vec q_{1/2} ) } {\vec p}_{1/2} \to \vec 0$.  When all entries of $\mathcal{G}$ are bounded,  $e^{-\lambda \mathcal{G}(\vec q_{1/2} ) }\to I_N$ as $\lambda \to 0$.  Therefore, $\vec p_{1/2}-\tilde{\vec p}_{1/2}\to\vec 0$ and $\vec Z(t)$ as defined in \cref{Z} is the solution of the Hamiltonian equation \cref{eq:ham}  on $[0,1]$ in the limit $\lambda\to 0$. Let $[\vec p_0,\vec q_0]=R S(1/2;0, R[\vec p_{1/2},\vec q_{1/2}])$ for $R[\vec p,\vec q]=[-\vec p,\vec q]$. Then,
\begin{align*}
\min F \to& \min \beta H(\vec p_{1/2},\vec q_{1/2})+0\\
&+\frac{1}{2\delta^2}\bp{\norm{\vec q^r-S_q(1/2;0,[-\vec p_{1/2}, \vec q_{1/2}])}^2+\norm{\vec q^t-S_q(1/2;0,[\vec p_{1/2}, \vec q_{1/2}])}}\\
=& \min \beta H(\vec p_{0},\vec q_{0})%
+\frac{1}{2\delta^2}\bp{\norm{\vec q^r-\vec q_0}^2+\norm{\vec q^t-S_q(1;0,[\vec p_0,\vec q_0])}}.
\end{align*}
Here we use the fact that $H$ is constant along solutions of \cref{eq:ham}. The last expression is the same as \cref{MAP1F}, as required
\end{proof}
\begin{corollary} If $\mathcal{G}(\vec q)$ is uniformly positive definite over $\vec q\in B^N\subset \real^{dN}$, then in the limit
$\beta\to\infty$, $\vec q_{1/2}$ converges to the arithmetic average $(\vec q^r+\vec q^t)/2$.\label[corollary]{cor_beta}
\end{corollary}
\begin{proof}
 Rescale the objective function
\begin{align*}
&\frac{1}{\beta}F(\vec p_{1/2}, \vec q_{1/2},\tilde{\vec p}_{1/2})%
\coloneq  H(\vec{p}_{1/2},\vec q_{1/2})%
+\frac {1}{4 \lambda}%
\norm{\tilde {\vec{p}}_{1/2}-e^{-\lambda
\mathcal{G}(\vec{q}_{1/2})} \vec{p}_{1/2}}^2\\%
&\qquad  %
+\frac{1}{2\beta\delta^2}%
\pp{%
	\norm{\vec q^r-S_q(1/2; 0, [-\vec  p_{1/2},\vec q_{1/2}] ) }^2%
		+\norm{\vec q^t-S_q(1/2; 0, [\tilde{\vec{p}}_{1/2},\vec{q}_{1/2}] ) }^2%
	}.
	\end{align*}
	This converges to zero as $\beta\to \infty$. Hence, $H(\vec p_{1/2},\vec q_{1/2})\to 0$, so that $\vec p_{1/2}\to\vec 0$ if $\mathcal{G}$ is uniformly positive definite. The second term implies that $\tilde{\vec p}_{1/2}\to \vec 0$. Then $\min F \to \frac{1}{2 \delta^2} (\Vpair{\vec q^r-\vec q_{1/2} }^2+\Vpair{\vec q^t-\vec q_{1/2}}^2)$.
\cref{lemma1} gives $\vec q_{1/2}$ is the arithmetic average.
		\end{proof}

The reverse limits are degenerate: As $\lambda\to\infty$,  $\tilde{\vec p}_{1/2}$ and $\vec{p}_{1/2}$ are not coupled and may be chosen independently. In particular, the second of the data terms can be always  be made zero. The remaining terms are minimised by taking $\vec q_{1/2}=\vec q^r$ and $\vec p_{1/2}=0$. For the limit as the noise grows and overwhelms the system,  $\beta\to 0$, there is no Hamiltonian or momenta coupling, and only data terms remain. Then $\vec q_{1/2}$ can be placed anywhere, as the momenta can be chosen arbitrarily without cost. This case  has a very shallow energy landscape and $\vec q_{1/2}$ is not well determined. Both these are outside the regime used in the derivation of the approximation \cref{eq:Ct}.

When the terms are balanced, the optimisation must achieve some
accuracy in flowing to the landmark points, coupling the
 momenta, and moderation of the energy in $H$. We see
in \cref{fig:yesav} examples where the arithmetic average and MAP average are very different.

\subsubsection{Computations with two landmark sets}
The MAP point can be found using unconstrained numerical
optimisation. The objective function is more complicated this time,
due to the matrix exponential
$e^{-\lambda \mathcal{G}(\vec q_{1/2})}$ and the required
derivative of the matrix exponential (for gradient-based optimisation methods). These functions are
available in Python's SciPy library, amongst others. The
Laplace method can be applied, again using Gauss--Newton
approximations and removal of negative eigenvalues, to
determine an approximation to the covariance matrix of the
posterior distribution.

To define an average of two sets of landmarks
$\vec q^{a,b}$, we choose $\vec q^r=\vec q^a$ and
$\vec q^t=\vec q^b$ and find the MAP point
$(\vec p_{1/2}, \vec q_{1/2}, \tilde{\vec p}_{1/2})$. The landmarks $\vec q^*=\vec q_{1/2}$ are used as the average of $\vec q^r$ and $\vec q^t$.  An example of the resulting average  is compared to the arithmetic average in \cref{fig:2av}.

\begin{figure}
  \centering
  \includegraphics{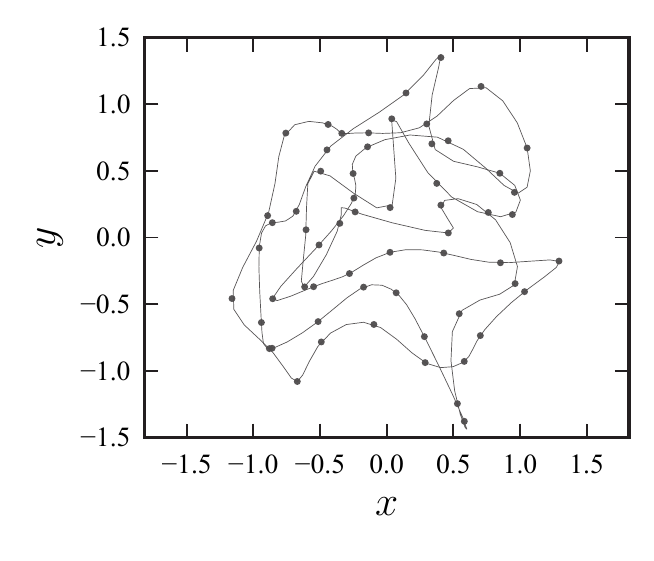}
  \includegraphics{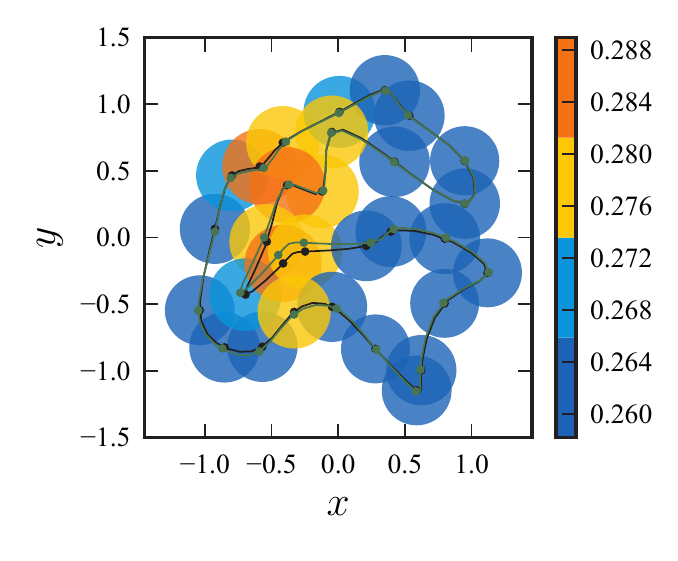}
  \caption{The left-hand plot shows two sets of
    landmarks. The right-hand plot shows two versions of the
    average landmarks. The black shape is calculated
    using the second splitting prior (with parameter
    $\lambda=0.1$, $\beta=25$) and assuming landmarks are
    known to $\Nrm(\vec 0, \delta^2 I)$ errors
    ($\delta^2=0.005$, so $\delta\approx 0.07$). The discs indicate one standard deviation of the posterior distribution (via the Laplace/Gauss--Newton approximation). The dark green shape is an arithmetic
    average.} \label{fig:2av}
\end{figure}

\subsubsection{Generalisation to multiple landmark sets}
We generalise the second splitting prior to allow for more landmark sets and thereby define an average of multiple landmark sets. Let $\vec q^*\in B\subset\real^{dN}$ be the desired average and let $\vec p^*\in\real^{dN}$ be an associated momenta. For the prior distribution, we assume $[\vec p^*,\vec q^*]$ follow the Gibbs distribution.
Let $\vec q^j\in B \subset \real^{dN}$  for $j=1,\dots,J$ denote the given data set of landmarks and associate to each  momenta $\vec p^j$.  We couple each $\vec p^j$ to $[\vec p^*,\vec q^*]$ via the time-one evolution of \cref{eq:ou}. With Gaussian errors in the approximation of the data $\vec q^j$ by the time-half evolution of the Hamiltonian system from $[\vec p^j,\vec q^*]$, this leads to the objective function for the MAP point:
\begin{gather}
\begin{split}
F(\vec p^*, \vec q^*, \vec p^j)%
&\coloneq \beta  H(\vec{p}^*,\vec q^*)%
+\frac {\beta}{4 \lambda}%
\sum_{j=1}^J \norm{ {\vec{p}}^j-e^{-\lambda
		\mathcal{G}(\vec{q}^*)} \vec{p}^*}^2\\%
&\qquad  %
+\frac{1}{2\delta^2}%
\sum_{j=1}^J
	\norm{\vec q^j-S_q(1/2;  0, [\vec  p^j,\vec q^*] ) }^2.
\end{split}\label{eq:multi}
\end{gather}
There are $J+1$ momenta and this objective does not reduce to \cref{eq:av2}, which depends on two momenta for $J=2$ landmark sets (see \cref{two}).  The limit as $\lambda\to 0$ is different and $\vec q^*$ cannot converge to the midpoint on the paths, as there is no such thing as a single flow between the landmark points for $J>2$. The extra momenta $\vec p^*$ is introduced as a substitute and provides a means of coupling the deformation for each landmark set to a single momentum. In contrast, as we now show, the limiting behaviour as $\beta\to \infty$ resembles the two-landmark average found by studying \cref{eq:av2}.
\begin{figure}
	\centering
	\includegraphics{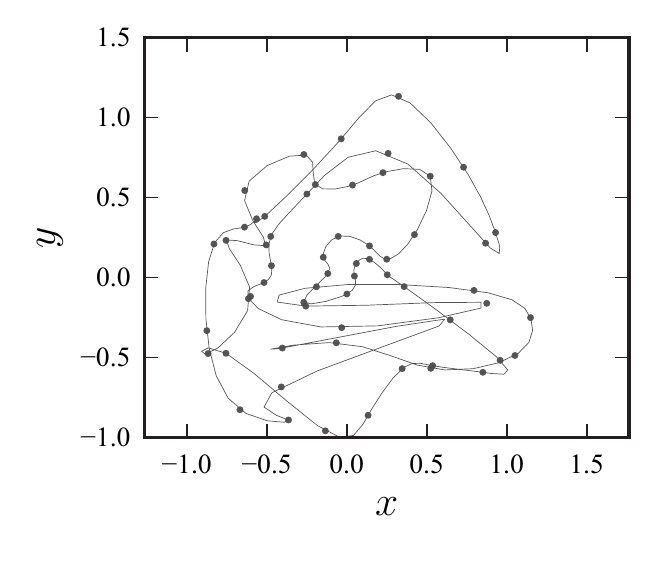}\\
	\includegraphics{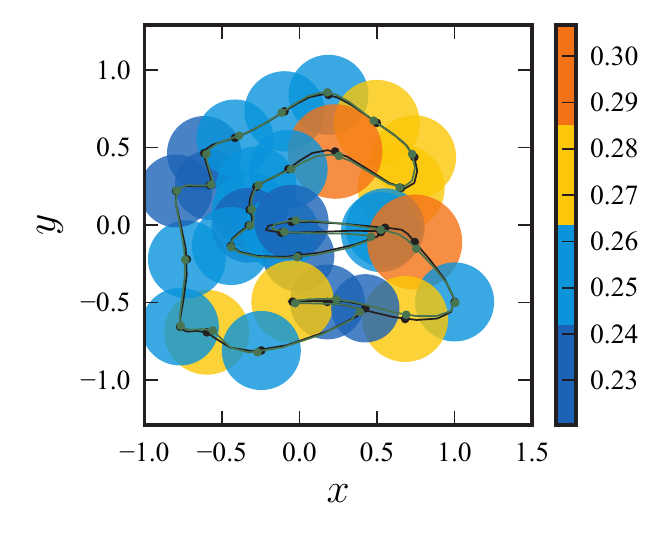}
	\includegraphics{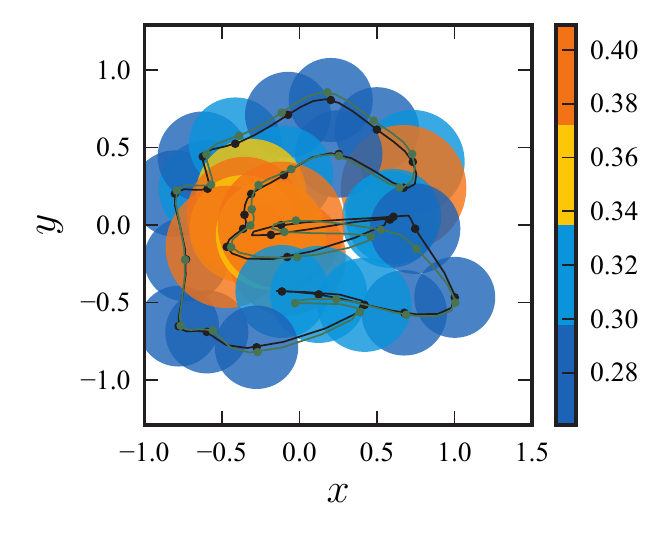}
	\caption{For the two sets of landmarks, the plots show averages (black lines) according to \cref{eq:av2} (left) and \cref{eq:multi} (right) with coloured discs showing one standard deviation. Both are close to the arithmetic average, shown in green, with the multi-set objective function being less close and having large standard deviations.} \label{two}
\end{figure}
\begin{theorem}\label{thm} Let $[\vec p^*,\vec q^*, \vec p^j]$ denote the minimiser of \cref{eq:multi}. Suppose that
	\begin{enumerate}
		\item $G(\vec q_i)$ is uniformly bounded over $\vec q_i\in B\subset \real^d$ and $\lambda\to 0$, or
		\item $\mathcal{G}(\vec q)$  is uniformly positive definite over $\vec q\in B^N\subset \real^{dN}$ and $\beta\to\infty$.
	\end{enumerate}
	In the limit,  $\vec q^*$ converges to the arithmetic average $(\vec q^1+\dots+\vec q^J)/J$.
\end{theorem}
\begin{proof} The argument for $\beta\to\infty$ is the same as \cref{cor_beta}. We concentrate on the case $\lambda\to 0$. By arguing similarly to \cref{cor_lam}, $\min F$ and
	$
	(\beta/\lambda)\Vpair{{\vec p}^j - e^{-\lambda \mathcal{G} (\vec q^*)} {\vec p}^*}^2
	$
	are bounded as $\lambda\to 0$.  Hence, ${\vec p}^j - e^{-\lambda \mathcal{G}(\vec q^* ) } {\vec p}^* \to \vec 0$ and, because entries of $\mathcal{G}$ are bounded,  we know that  $\vec p^j-{\vec p}^*\to\vec 0$. We can minimise the two remaining terms separately: $\beta H(\vec p^*,\vec q^*)$ is minimised by $\vec p^*=\vec 0$ and the data term is minimised when $S_q(1/2; 0,[\vec p^j,\vec q^*])$ equals the arithmetic average. This is achieved when $\vec p^j=\vec p^*=\vec 0$ and $\vec q^*$ is the arithmetic average.
\end{proof}
The methodology for this objective are similar to \cref{eq:av2}: the minimum is found by unconstrained numerical optimisation and $\vec q^*$ is used as an average. The Hessian can be evaluated at the MAP point to define an approximate posterior covariance matrix.

An example of the resulting average for sixteen samples is compared to the arithmetic average in \cref{fig:tenav}. The standard deviation is reduced in comparison to \cref{fig:2av},  from the range $[0.26,0.29]$ down to $[0.15,0.18]$, which is roughly a factor $1.6$ decrease from a factor eight increase in the number of samples, and less than expected from the central limit theorem. \cref{fig:yesav} shows computations of 64 and 256 samples from the same distribution of landmark sets. The distinction between arithmetic and MAP averages is even stronger. The standard deviations reduce but again moderately compared to the factor of two  expected from a factor four increase in the number of samples.

The final example in \cref{fig:rockav} shows how the MAP average moves closer to the arithmetic average when the value of $\beta$ is increased from $\beta=50$ to $\beta=100$, as discussed in \cref{thm}.

\begin{figure}
	\centering
	\includegraphics{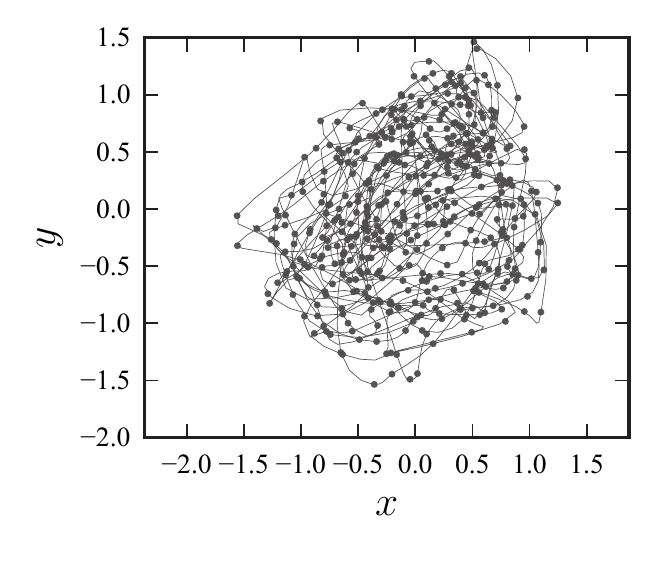}
	\includegraphics{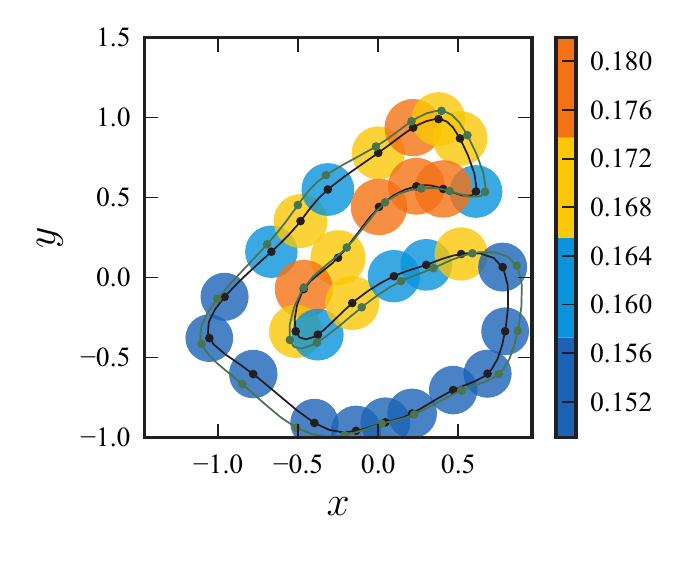}
	\caption{Similar to \cref{fig:2av}, except sixteen landmark sets are taken and the MAP average is computed using the objective function \cref{eq:multi}.} \label{fig:tenav}
\end{figure}
\begin{figure}
	\centering
	\includegraphics{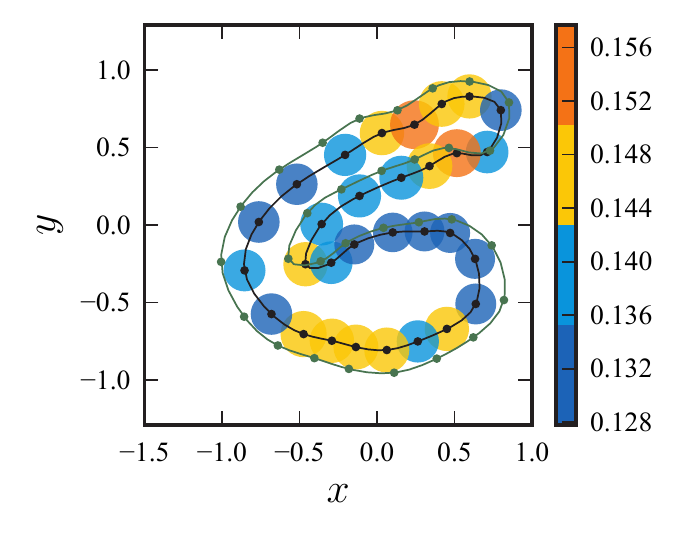}
	\includegraphics{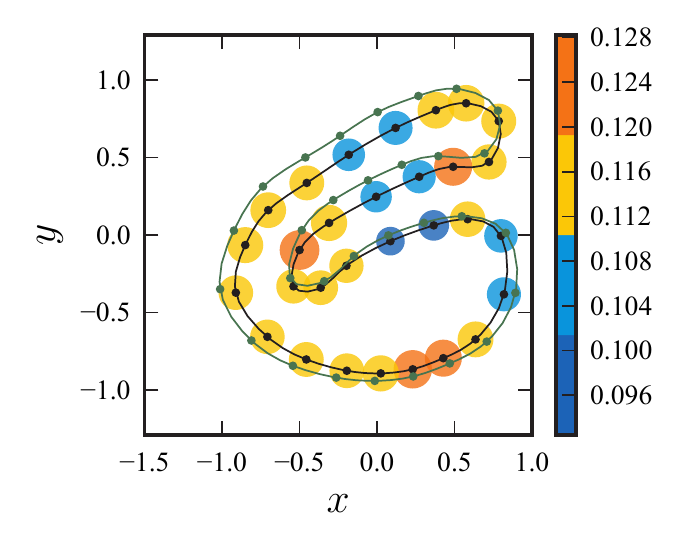}\\
	\includegraphics{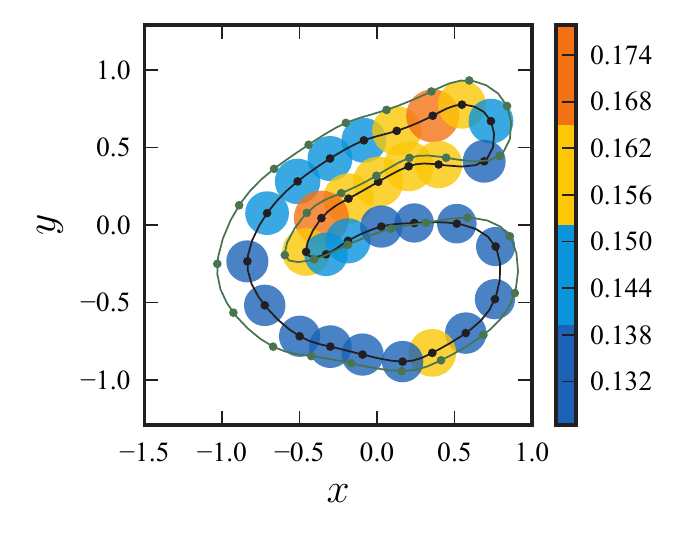}
	\includegraphics{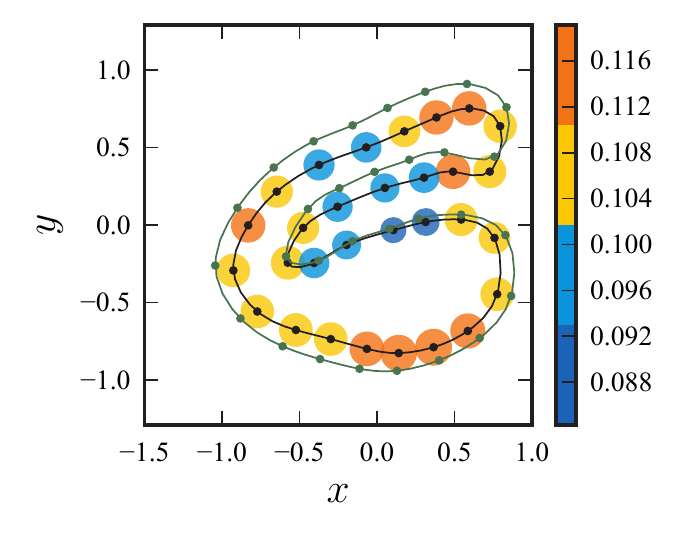}
	\caption{Here, we show two computations for the average of 64 (left) and 256 (right) independent samples, using the second splitting prior with $\lambda=0.1$ and $\beta=25$ (black line) and the arithmetic average (green line). The rows are calculations of the same averages for independent samples. The colours indicate one standard deviation of the computed posterior distribution.} \label{fig:yesav}
\end{figure}
\begin{figure}
	\centering
	\includegraphics{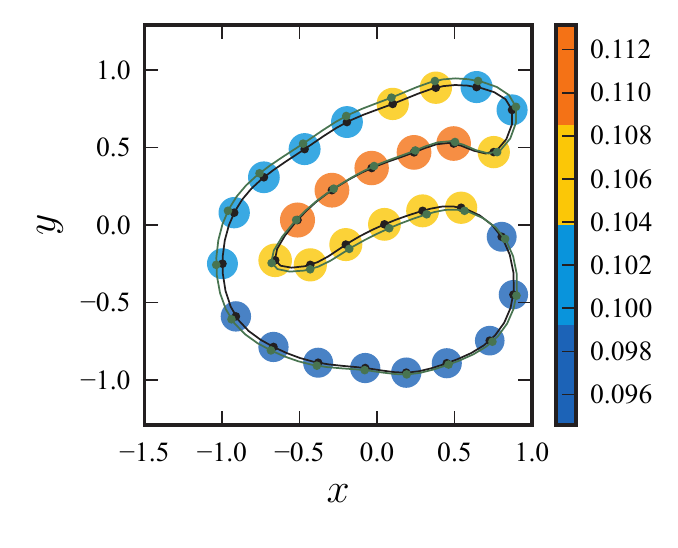}
	\includegraphics{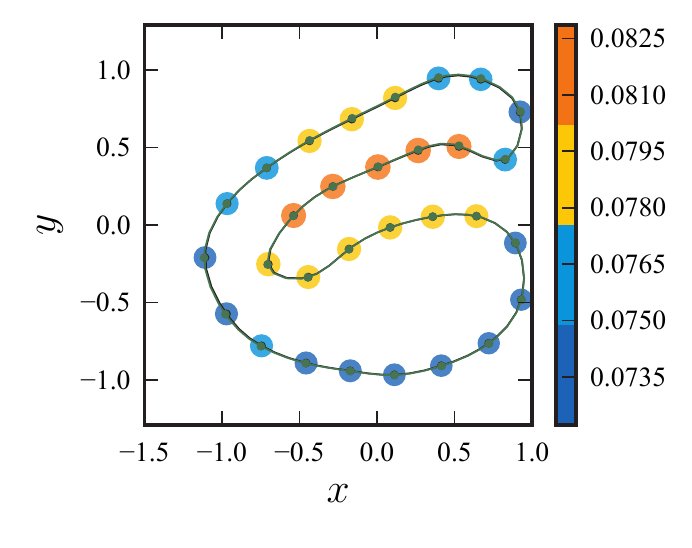}
	\caption{The plots show the averages (black) provided by the MAP point for $\lambda=0.1$ with $\beta=50$ (left) and $\beta=100$ (right) in comparison to the arithmetic average (green) for 64 landmark sets. As we shown in \cref{thm}, the averages become closer as $\beta$ is increased.} \label{fig:rockav}
\end{figure}
\section{Conclusion}
This article introduces a type of Langevin equation for performing image registration by landmarks in the presence of uncertainty. The Langevin equation is used to define a prior distribution on the set of diffeomorphisms. It is computationally difficult to sample the diffusion bridge for the Langevin equation. To allow computation, we introduced three approximate prior distributions: the first by linearising the Langevin equation about the solution of a Hamiltonian problem, the others by making an operator splitting to the generator and using a Baker--Campbell--Hausdorff-type approximation to the solution of the Fokker--Planck equation. We give computational examples using the MAP point and Laplace method to find approximate variances for the posterior distribution.

The second splitting prior lends itself to formulating an average of two landmark sets. We defined the average of two landmark sets via the prior and studied the limits as the  inverse temperature $\beta\to\infty$ (corresponding to the arithmetic average) and dissipation $\lambda\to 0$ (corresponding to the midpoint of the registration identified by the MAP point for the first splitting prior). This was extended to define an average for multiple landmark sets, with examples provided for both two and multiple landmark sets.

The work was limited by the current technology for sampling hypoelliptic diffusion bridges, and it will be interesting to see how this area develops.

Another avenue of future work is incorporating invariants into the prior distribution, such as conservation of average landmark position. The Langevin equation can be adjusted so that the dynamics live on a subspace of $\real^{2dN}$ where the Gibbs  distribution may be a probability measure and landmark average is  invariant. The following variation of \cref{eq:11} has invariant measure $\exp(-\beta H)$ and satisfies $\frac{d}{dt}\sum \vec p_i=\vec 0$ for isotropic $G$:
\begin{gather}
\begin{split}
  d\vec p_i%
  &= \Big[-\lambda \sum_{j\ne i} w(q_{ij})^2 \hat {\vec q}_{ij} \hat{ \vec
    q}_{ij}\cdot \nabla_{\vec p_i} H%
  - \nabla_{\vec    q_i}H\Big]\,dt%
  +
  \sigma  \sum_{j\ne i} w(q_{ij}) \hat {\vec q}_{ij}  dW_{ij}(t),\\
  \frac{d\vec q_i}{dt}%
  &=  \nabla_ {\vec p_i}H.
\end{split}  \label{eq:11b}
\end{gather}
Here $\hat{\vec q}_{ij}$ is the inter-particle unit vector
and $q_{ij}=\Vpair{\vec q_i-\vec q_j}$.  This time, $
W_{ij}(t)$ are \iid scalar Brownian motions for $i<j$ and
$W_{ij}=W_{ji}$.  Here $w\colon\real\to \real^+$ is a
coefficient function, which could be identically equal to
one for simplicity.  For given $\bar{\vec p}\in\real^d$, we see
$\exp(-\beta H)$ is an invariant measure on the subspace of
$\real^{2dN}$ with $\frac{1}{N}\sum_{i=1}^N\vec
p_i=\bar{\vec p}$ (the centre of mass is invariant for
$\bar{\vec p}=\vec 0$). This can be shown to be invariant by
using the Fokker--Planck equation as above, with $\lambda$
and $\sigma$ replaced by position-dependent coefficients that
still cancel out under the fluctuation--dissipation
relation. See \cite{dpd,Shardlow2003-qo,shardlow04:_geomet}.

\appendix
\section{Linearised equations}\label{sec:lin_dist}

 We write down equations to compute the mean and
  covariance, using backward and forward Euler
  approximations.  Suppose that  $\vec \delta_{n_1} \sim \Nrm ( \vec 0, C_1)$,
  for some given $C_1$. We wish to calculate the joint
  distribution of $\vec\delta_{n}$ for
  $n=0,\dots,N_{\tstep}$. This is easy to do as the joint
  distribution is Gaussian and we derive update rules for
  the mean and covariance:  From
  \[
  \vec \delta_{n+1} %
  = M^+_n \vec\delta_n + \vec A_n %
  + \begin{pmatrix} \sigma \Delta \vec{W}_n \\ %
    0 \end{pmatrix},
  \]
  we get an update rule for the mean
  \begin{align*}
    \vec\mu_{n+1}%
    = \mean{ \vec\delta_{n+1}} %
    = M^+_n\vec \mu_n +  \vec A_n.
  \end{align*}
  Similarly, when time-stepping backwards,
  \[
  \vec\mu_{n-1}%
  = \mean{ \vec\delta_{n-1}} %
  = M^-_n\vec \mu_n+ \vec A_n.
  \]
  For the covariance update along the diagonal moving forward,
  \begin{align*}
    \mean{\vec\delta_{n+1}\vec\delta_{n+1}^\trans} %
    &= \mean{\pp{M^+_n \vec\delta_n + \vec A_n \strutB} %
      \pp{M^+_n \vec\delta_n + \vec A_n \strutB}^\trans}%
    + \begin{pmatrix}\sigma h I_{dN} & 0 \\%
      0 &  0 \end{pmatrix}  \\
    &= M^+_n
    \mean{\vec\delta_n \vec\delta_n^\trans}%
    M^{+\trans}_n%
    +  \vec A_n { \vec\mu_{n+1}^\trans}%
    + \vec \mu_{n+1}\vec A_n^\trans %
    - \vec A_n \vec A_n^\trans%
    + \begin{pmatrix}\sigma h I_{dN} & 0 \\0& 0 \end{pmatrix}.
  \end{align*}
  Similarly, moving backwards,
  \begin{align*}
    \mean{\vec\delta_{n-1}\vec\delta_{n-1}^\trans} %
    &= \mean{\pp{M^-_n \vec\delta_n + \vec A_n \strutB} %
      \pp{M^-_n \vec\delta_n + \vec A_n \strutB}^\trans}%
    + \begin{pmatrix}\sigma h I_{dN} & 0 \\%
      0 &  0 \end{pmatrix}  \\
    &= M^-_n \mean{\vec\delta_n \vec\delta_n^\trans}%
    M^{-\trans}_n%
    + \vec A_n { \vec\mu_{n-1}^\trans}%
    + \vec \mu_{n-1}\vec A_n^\trans %
    - \vec A_n \vec A_n^\trans%
    + \begin{pmatrix}\sigma h I_{dN} & 0 \\%
      0& 0 \end{pmatrix}.
  \end{align*}
  The remaining parts of the matrix $\mean{\vec \delta_j
    \vec \delta_k^\trans}$ can be computed by sideways
  moves along either a row or column using the rules: if
  $k\ge j$,
  \begin{align*}
    \mean{\vec\delta_{j}\vec\delta_{k+1}^\trans}%
    &= \mean{ \vec\delta_j \vec\delta_k^\trans}M_k^{+\trans}%
    + \vec\mu_j \vec A_k^\trans,\\
    \mean{\vec\delta_{k+1}\vec\delta_{j}^\trans}%
      &= M^+_k\mean{ \vec\delta_k \vec\delta_j^\trans}%
      + \vec A_k\vec\mu_j^\trans,
    \end{align*}
    and if $k\le j$
    \begin{align*}
      \mean{\vec\delta_{j}\vec\delta_{k-1}^\trans} %
      &= \mean{ \vec\delta_j
        \vec\delta_k^\trans}M^{-\trans}_k%
      + \vec\mu_j \vec A_k^\trans,\\
      \mean{\vec\delta_{k-1}\vec\delta_{j}^\trans}%
      &= M^-_k\mean{ \vec\delta_k \vec\delta_j^\trans}%
      + \vec A_k\vec\mu_j^\trans.
    \end{align*}
    Finally,
    $
    \cov (\vec\delta_j, \vec\delta_n)%
    = \mean{\vec\delta_j \vec\delta_n^\trans}%
    -\vec\mu_j \vec\mu_n^\trans$.

  \def\bibfont{\normalfont\small}
\ifarxiv
 \bibliographystyle{siamplain}
\else
 \bibliographystyle{siamplain}
\fi
\bibliography{sde_imag}
\end{document}